\documentclass[11pt]{article}
\usepackage{amsmath,amsfonts,amsthm,amssymb, mathtools}
\usepackage{mathrsfs, graphicx,color,latexsym, tikz, calc,
}
\usepackage{tikz-cd}
\usepackage{graphicx}
\usepackage{epstopdf}
\usepackage{enumerate}
\usepackage{caption}
\usepackage{stmaryrd}
\usepackage{hyperref}

\usetikzlibrary{shadows}
\usetikzlibrary{patterns,arrows,decorations.pathreplacing}
\usepackage{ulem}

\numberwithin{equation}{section}
\newtheorem{theorem}{Theorem}[section]
\newtheorem{lemma}[theorem]{Lemma}
\newtheorem{prop}[theorem]{Proposition}

\newtheorem{problem}[theorem]{Problem}

\allowdisplaybreaks[4]

\date{\today}
\title{\bf\Large Maximal intersecting families revisited\footnote{This paper was published on Discrete Mathematics 349 (2026) 114654. This is the final version. Following reviewer's
suggestion, we have changed the original title ``Stabilities of intersecting families revisited''.  E-mail addresses: \url{wuyjmath@163.com} (Y. Wu), \url{ytli0921@hnu.edu.cn} (Y. Li), \url{fenglh@163.com} (L. Feng), \url{jiuqiang68@126.com} (J. Liu), \url{yuguihai@126.com} (G. Yu).}}
\author{
{\small  Yongjiang Wu$^a$, Yongtao Li$^a$, Lihua Feng$^a$, Jiuqiang Liu$^b$, Guihai Yu$^{c,}$\footnote{Corresponding author}}\\[2mm]
\small $^a$School of Mathematics and Statistics, HNP-LAMA, Central South University\\
 \small Changsha, Hunan, 410083, China\\
\small $^b$Department of Mathematics, Eastern Michigan University\\
 \small  Ypsilanti, MI, 48197, USA \\ 
  \small $^c$College of Big Data Statistics, Guizhou University of Finance and Economics\\
 \small  Guiyang, Guizhou, 550025, China }

\begin{document}

\maketitle

\begin{abstract} 
The well-known Erd\H{o}s--Ko--Rado theorem 
states that for $n> 2k$, every intersecting family of $k$-sets of $[n]:=\{1,\ldots ,n\}$ has at most $ {n-1 \choose k-1}$ sets, and the extremal family consists of all $k$-sets containing a fixed element (called a full star). 
The Hilton--Milner theorem provides a stability result by determining the maximum size of a uniform intersecting family that is not a subfamily of a full star. 
  Further stability results were studied by Han and Kohayakawa (2017) and Huang and Peng (2024).  
Two families $\mathcal{F}$ and $\mathcal{G}$ are called cross-intersecting if for every $F\in \mathcal{F}$ and $G\in \mathcal{G}$, the intersection $F\cap G$ is non-empty.  
Let $k \geq 1, t\ge 0$ and $n \geq 2 k+t$ be integers.  Frankl (2016) proved that if $\mathcal{F} \subseteq\binom{[n]}{k+t}$ and $\mathcal{G} \subseteq\binom{[n]}{k}$ are cross-intersecting families, and $\mathcal{F}$ is non-empty and $(t+1)$-intersecting, then
$|\mathcal{F}|+|\mathcal{G}| \leq\binom{n}{k}-\binom{n-k-t}{k}+1$. Recently, Wu (2023) sharpened Frankl's result by establishing a stability variant.  
The aim of this paper is two-fold. 
Inspired by the above results, 
we first prove a further stability variant that generalizes both Frankl's result and Wu's result. 
Secondly, as an interesting application, 
we illustrate that the aforementioned results on cross-intersecting families could be used to establish the stability results of the Erd\H{o}s--Ko--Rado theorem. More precisely, we present new short proofs of the Hilton--Milner theorem, the Han--Kohayakawa theorem and the Huang--Peng theorem. 
Our arguments are more straightforward, and it may be of independent interest. 
\end{abstract}

 

\section{Introduction}

For any two integers $a\leq b$, let  $[a, b]=\{a, a+1, \ldots, b\}$ and  let $[n]=[1,n]$. Let $2^{[n]}$ denote the power set of $[n]$. For any $0 \leq k \leq n$, let $\binom{[n]}{k}$ denote the collection of all its $k$-element subsets. A family $\mathcal{F} \subseteq 2^{[n]}$ is called $k$-\textit{uniform} if $\mathcal{F} \subseteq \binom{[n]}{k}$. A family $\mathcal{F} \subseteq 2^{[n]}$ is called $t$-\textit{intersecting} if $|F\cap F^{\prime}|\geq t$
 for all $F, F^{\prime}\in \mathcal{F}$. If $t=1$, $\mathcal{F}$ is simply called \textit{intersecting}. Two families  $\mathcal{F}, \mathcal{G}\subseteq 2^{[n]}$ are  said to be \textit{isomorphic} if there exists a permutation $\sigma$ on $[n]$ such that $\mathcal{G}=\left\{\{\sigma(x):x\in F\}: F\in \mathcal{F}\right\}$, and we write $\mathcal{F}\cong\mathcal{G}$.

In 1961, Erd\H{o}s, Ko and Rado \cite{E61} determined the maximum size of $k$-uniform
intersecting families by proving that if $k\ge 1$, $n\ge 2k$ and 
 $\mathcal{F} \subseteq\binom{[n]}{k}$ is an intersecting family, then 
\begin{equation*} \label{eq-EKR-bound}
|\mathcal{F}|\leq\binom{n-1}{k-1}.
\end{equation*}
For $n > 2 k$, the equality holds if and only if $\mathcal{F}=\big\{F\in \binom{[n]}{k}: x\in F\big\}$ for some $x \in [n]$. 
There are many different proofs and methods for proving the Erd\H{o}s--Ko--Rado theorem; see, e.g., the Katona  circle method \cite{Kat1972},  the probabilistic method \cite{AS2016}, the algebraic methods \cite{Fur2006,HZ2017,Lov1979} and other combinatorial methods \cite{Day1974b,FF2012,H18,K64,KZ2018}.  
 For more results on extremal set theory, we refer the interested readers to the  comprehensive surveys \cite{FT16,Ellis2021}.

\subsection{Cross-intersecting families}

Two families  $\mathcal{F},\mathcal{G}\subseteq 2^{[n]}$ are  called \textit{cross-intersecting} if $|F\cap G|\geq 1$
 for any $F\in \mathcal{F}$ and $G\in \mathcal{G}$.

In 1967, Hilton and Milner \cite{H67} gave the following result. 

\begin{theorem}[Hilton--Milner \cite{H67}]  \label{mainH}
Let $k \geq 1$ and $n \geq 2 k$ be positive integers.  Let $\mathcal{F} \subseteq\binom{[n]}{k}$ and $\mathcal{G} \subseteq\binom{[n]}{k}$ be  non-empty cross-intersecting families. Under the condition that $|\mathcal{G}| \geq |\mathcal{F}|$, then
$$
|\mathcal{F}|+|\mathcal{G}| \leq\binom{n}{k}-\binom{n-k}{k}+1.
$$
For $n > 2 k$, the equality holds if and only if $\mathcal{F}=\left\{F_1\right\}$ for some $F_1 \in\binom{[n]}{k}$ and $\mathcal{G}= \big\{G \in\binom{[n]}{k}:$ $ G \cap F_1 \neq \emptyset \big\}$, or one more possibility when $k=2, \mathcal{F}=\mathcal{G}\cong\left\{G \in\binom{[n]}{2}: 1\in G\right\}$.
\end{theorem}

This result initiated the study of finding the maximum of the sum of sizes of cross-intersecting families. 
There are many generalizations of the Hilton--Milner result on cross-intersecting families; see, e.g.,  \cite{F16,F24,FT92,F98,FW2024-EUJC,GXZ2024,SFQ2022,WZ2013}. 
In particular, Frankl \cite{F16} gave the following inequality, which was applied to establish the stability result of Katona theorem \cite{K64}.

\begin{theorem}[Frankl \cite{F16}] 
\label{F16}
Let $k \geq 1, t\ge 0$ and $n \geq 2 k+t$ be integers.  Let $\mathcal{F} \subseteq\binom{[n]}{k+t}$ and $\mathcal{G} \subseteq\binom{[n]}{k}$ be  cross-intersecting families. If $\mathcal{F}$ is $(t+1)$-intersecting and $|\mathcal{F}| \geq 1$, then
$$
|\mathcal{F}|+|\mathcal{G}| \leq\binom{n}{k}-\binom{n-k-t}{k}+1.
$$
For $n > 2 k+t$, the equality holds if and only if $\mathcal{F}=\left\{F_1\right\}$ for some $F_1\in\binom{[n]}{k+t}$ and $\mathcal{G}=\left\{G \in\binom{[n]}{k}: G \cap F_1 \neq \emptyset \right\}$, or  one more possibility when $k=2, \mathcal{F}=\{[t+1] \cup\{i\}: i \in\{t+2, t+3, \ldots, n\}\}$ and $\mathcal{G}=\left\{G \in\binom{[n]}{k}: G \cap[t+1] \neq \emptyset\right\}$ under isomorphism.
\end{theorem}

In the case $t=1$, Bulavka and Woodroofe \cite{BW2024} told us a variant by replacing the $2$-intersecting with a condition on the shadow of $\mathcal{F}$. In fact, their result can be deduced from  Theorem \ref{F16}.  
Recently, the second author and Wu \cite{L24} sharpened   Theorem \ref{F16} in the case  $|\mathcal{F}|\ge 2$ and $t=1$ and used this to establish stability results for non-uniform $t$-intersecting families.
Subsequently, Wu \cite{W23} showed the following more general extension under the constraint $|\mathcal{F}|\ge 2$.

\begin{theorem}[Wu \cite{W23}] \label{W23}
Let $k \geq 3, t\ge 0$ and $n \geq 2 k+t$ be integers. Let $\mathcal{F} \subseteq\binom{[n]}{k+t}$ and $\mathcal{G} \subseteq\binom{[n]}{k}$ be cross-intersecting families. If $\mathcal{F}$ is $(t+1)$-intersecting and $|\mathcal{F}| \geq 2$, then
$$
|\mathcal{F}|+|\mathcal{G}| \leq\binom{n}{k}-\binom{n-k-t}{k}-\binom{n-k-t-1}{k-1}+2.
$$
For $n > 2 k+t$, the equality holds if and only if $\mathcal{F}=\left\{F_1, F_2\right\}$ for some $F_1, F_2 \in\binom{[n]}{k+t}$ with $\left|F_1 \cap F_2\right|=k+t-1$ and $\mathcal{G}=\left\{G \in\binom{[n]}{k}: G \cap F_1 \neq \emptyset \text{ and } G \cap F_2 \neq \emptyset\right\}$, or two more possibilities when $k=3$, namely, $\mathcal{F}=\{[t+2] \cup\{i\}: i \in\{t+3, t+4, \ldots, n\}\}\text{ and } \mathcal{G}=\left\{G \in\binom{[n]}{3}: G \cap[t+2] \neq \emptyset\right\}$, or $\mathcal{F}=\left\{F\in\binom{[n]}{t+3}: [t+1] \subseteq F\right\}\text{ and }  \mathcal{G}=\left\{G \in\binom{[n]}{3}: G \cap[t+1] \neq \emptyset\right\}$ under isomorphism.
\end{theorem}

Inspired by Theorems \ref{F16} and \ref{W23}, we investigate the further stability for family $\mathcal{F}$ with  $|\mathcal{F}| \geq 3$. 
The main result of our paper  is presented below. 

\begin{theorem}[Main result]  \label{51}
Let $k \geq 4, t\ge 0$ and $n \geq 2 k+t$ be integers. Let $\mathcal{F} \subseteq\binom{[n]}{k+t}$ and $\mathcal{G} \subseteq\binom{[n]}{k}$ be cross-intersecting families. If $\mathcal{F}$ is $(t+1)$-intersecting and $|\mathcal{F}| \geq 3$, then
$$
|\mathcal{F}|+|\mathcal{G}| \leq\binom{n}{k}-\binom{n-k-t}{k}-\binom{n-k-t-1}{k-1}-\binom{n-k-t-2}{k-2}+3.
$$
For $n > 2 k+t$, the equality holds if and only if $\mathcal{F}=\left\{F_1, F_2, F_3\right\}$ with $F_i:=[k+t-1]\cup \{k+t-1+i\}$ for each $i\in [3]$  and $\mathcal{G}=\left\{G \in\binom{[n]}{k}: G \cap F_i \neq\emptyset \text{ for any } i=1, 2, 3\right\}$, or  one more possibility when $k=4, \mathcal{F}=\{[t+3] \cup\{i\}: i \in\{t+4, \ldots, n\}\}$ and $\mathcal{G}=\left\{G \in\binom{[n]}{4}: G \cap[t+3] \neq \emptyset\right\}$ under isomorphism.
\end{theorem}

In contrast to the Erd\H{o}s--Ko--Rado theorem and the Hilton--Milner theorem (see Theorem \ref{H67}), which gives an upper  bound smaller than $k{n-2 \choose k-2}$. This differs by an order of magnitude of $n$ with the bound ${n-1 \choose k-1}$. 
While, all the upper bounds in Theorems \ref{F16}, \ref{W23} and \ref{51} have the same leading term (asymptotically) and only differ in the smaller order terms. It is interesting that this difference in smaller order terms will be enough to establish the stability results in Section \ref{sec-1-2}.

\medskip 
\noindent 
{\bf Organization.} 
The proof of Theorem \ref{51} will be provided in Section \ref{sec-2}. 
Before showing the proof, we would like to illustrate the 
significance of Theorem \ref{51} for 
the stability results of the Erd\H{o}s--Ko--Rado theorem 
in the next section. 
More precisely,  
we show that the above results on cross-intersecting families can be applied to prove the celebrated Hilton--Milner theorem, the Han--Kohayakawa theorem as well as the Huang--Peng theorem; see the forthcoming Theorems \ref{H67}, \ref{H17} and \ref{H24}, respectively. 
The second main contribution of this paper is to provide new methods to prove these stability results. 
The detailed discussions will be presented in Section \ref{sec-3}. 
To ensure our presentation clearly, some tedious arguments will be postponed to the Appendix.

\subsection{Applications: Stability for Erd\H{o}s--Ko--Rado's theorem}

\label{sec-1-2}

In this section, we show that our result is extremely useful for the study of stability of $k$-uniform intersecting families. Indeed, we will provide a simple unified approach to proving several important stability results for the Erd\H{o}s--Ko--Rado theorem; see, e.g., \cite{H67,H17,H24}. 

A family $\mathcal{F}$ that satisfies  $\mathcal{F}=\big\{F\in \binom{[n]}{k}: x\in F\big\}$ is called a \textit{full star} centered at $x$, and $\mathcal{F}$ is called \textit{EKR} or a \textit{star} if
$\mathcal{F}$ is contained in a full star. Such family is a ``trivial" example of intersecting family. If $\mathcal{F}$ is not EKR, then  $\cap_{F\in\mathcal{F}}F=\emptyset$ and $\mathcal{F}$ is called \textit{non-trivial}.

In 1967,  Hilton and Milner \cite{H67} determined the maximum size of a non-trivial $k$-uniform intersecting family. 
We denote 
$\mathcal{HM}(n, k)=\big\{F \in\binom{[n]}{k}: 1 \in F, F \cap[2, k+1] \neq \emptyset\big\} \cup\{[2, k+1]\}$. In the case $k=3$, we denote 
$\mathcal{T}(n, 3)=\big\{F \in\binom{[n]}{3}:|F \cap[3]| \geq 2\big\}$. 

\begin{theorem}[Hilton--Milner \cite{H67}]
\label{H67}
Let $k \geq 2$ and $n \geq 2 k+1$ be integers. Let $\mathcal{F} \subseteq\binom{[n]}{k}$ be an intersecting family.  
If $\mathcal{F}$ is not EKR, then
$$
|\mathcal{F}|\leq\binom{n-1}{k-1}-\binom{n-k-1}{k-1}+1.
$$
For $k=2$ or $k\ge 4$, 
the equality holds if and only if $\mathcal{F}$ is isomorphic to $ \mathcal{HM}(n, k)$; 
For $k=3$, the equality holds if and only if $\mathcal{F}$ is isomorphic to $ \mathcal{HM}(n,3)$ or $ \mathcal{T}(n, 3)$.
\end{theorem}

A family $\mathcal{F}$ is called \textit{HM} if
$\mathcal{F}$ is isomorphic to a subfamily of $\mathcal{HM}(n, k)$.  In 2017, Han and Kohayakawa \cite{H17} took stability further and determined the maximum $k$-uniform intersecting family that is neither EKR nor HM, and not a subfamily of $\mathcal{T}(n, 3)$ if $k=3$. We shall mention that for $k\geq 4$, this was already solved by Hilton and Milner \cite{H67} in 1967. Before introducing their result, let us  describe some set families.

Let $k \geq 3$ and $n \geq 2 k+1$ be integers. 
For $ i \in[1, k-1]$, $E \in \binom{[n]}{k-1}$ and $x_0 \in [n] \setminus E$, let $J_i$ be a subset of $[n]\setminus (E\cup \{x_0\})$ with $|J_i|=i$, we define 
$$
\mathcal{J}_i(n, k)=\left\{F \in\binom{[n]}{k}: x_0 \in F, F \cap (E\cup\{j\}) \neq \emptyset \text{ for each } j \in J_i\right\} \cup \Bigl\{E \cup\{j\}: j \in J_i\Bigr\} .
$$
Note that $\mathcal{J}_1(n,k)\cong\mathcal{HM}(n,k)$.  
For $i \in[2, k]$, $E \in \binom{[n]}{i}$ and $x_0 \in[n]\backslash E$, we define 
$$
\mathcal{G}_i(n, k)=\left\{G \in\binom{[n]}{k}: E \subseteq G\right\} \cup\left\{G \in\binom{[n]}{k}: x_0 \in G, G \cap E \neq \emptyset\right\} .
$$
Note that $\mathcal{G}_2(n, 3)$ is isomorphic to $\mathcal{T}(n, 3)$.

\begin{theorem}[Han--Kohayakawa \cite{H17}] 
\label{H17}
	Let $k \geq 3$ and $ n\geq 2 k+1$ be integers. Let $\mathcal{F} \subseteq\binom{[n]}{k}$ be an intersecting family. If $\mathcal{F}$ is neither EKR nor HM, and $\mathcal{F} \nsubseteq \mathcal{G}_2(n, 3)$ for $k=3$. Then
$$
|\mathcal{F}|\leq\binom{n-1}{k-1}-\binom{n-k-1}{k-1}-\binom{n-k-2}{k-2}+2.
$$
 For $k=3$ or $k\ge 5$, the equality holds if and only if $\mathcal{F}$ is isomorphic to  $\mathcal{J}_2(n, k)$; 
  For $k=4$, the equality holds if and only if $\mathcal{F} $ is isomorphic to $ \mathcal{J}_2(n, 4), \mathcal{G}_2(n, 4)$ or $\mathcal{G}_3(n, 4)$.
\end{theorem}

Han and Kohayakawa \cite{H17} also asked what is the maximum $k$-uniform intersecting family which is neither EKR nor HM, and not a subfamily of $\mathcal{J}_2(n, k)$,  and not a subfamily of $\mathcal{G}_2(n, 4)$, $\mathcal{G}_3(n, 4)$ if $k=4$. In 2017, Kostochka and Mubayi \cite{KM17} answered this question for large enough $n$. In 2024, Huang and Peng \cite{H24} completely solved this question for all $n\geq 2k+1$.

To construct a large intersecting family that is not EKR nor HM, note that $\mathcal{H}=\mathcal{HM}(n,k)$ has the property that $\mathcal{H}(\bar{1})=\{[2,k+1]\}$. As such, consider families $\mathcal{F}\subseteq\binom{[n]}{k}$ such that $|\mathcal{F}(\bar{1})|\geq 2$. Moreover, we want $\mathcal{F}(\bar{1})$ to be $\ell$-intersecting for $\ell$ as large as possible. When $\ell=k-1$, this gives arise to $\mathcal{J}_{i}(n,k)$. When $\ell=k-2$, this gives arise to $\mathcal{K}_{2}(n,k)$ where  $\mathcal{K}_2(n, k)$ is defined as follows.
For any $E_1, E_2 \in \binom{[n]}{k}$ with $|E_1\cap E_2|=k-2$, and $x_0\in [n]\backslash (E_1\cup E_2)$, let
$$
\mathcal{K}_2(n, k)=\left\{G \in\binom{[n]}{k}: x_0\in G,  G \cap E_1 \neq \emptyset \text{ and } G \cap E_2 \neq \emptyset\right\}\cup\{E_1, E_2\}.
$$

\begin{theorem}[Huang--Peng \cite{H24}] \label{H24}
 Let $k \geq 4$ and $\mathcal{F} \subseteq\binom{[n]}{k}$ be an intersecting family which is neither EKR nor HM, and $\mathcal{F} \nsubseteq \mathcal{J}_2(n, k)$, in addition, $\mathcal{F} \nsubseteq \mathcal{G}_2(n, 4)$ and $\mathcal{F} \nsubseteq \mathcal{G}_3(n, 4)$ if $k=4$.

\noindent(i) If $2 k+1 \leq n \leq 3 k-3$, then
$$
|\mathcal{F}| \leq\binom{n-1}{k-1}-2\binom{n-k-1}{k-1}+\binom{n-k-3}{k-1}+2.
$$
For $k\ge 5$, 
the equality holds if and only if  $\mathcal{F}$ is isomorphic to $ \mathcal{K}_2(n, k)$;
For $k=4$, the equality holds if and only if   $\mathcal{F} $ is isomorphic to $ \mathcal{K}_2(n, 4)$ or $\mathcal{J}_3(n, 4)$.

\noindent(ii) If $n \geq 3 k-2$, then
$$
|\mathcal{F}| \leq\binom{n-1}{k-1}-\binom{n-k-1}{k-1}-\binom{n-k-2}{k-2}-\binom{n-k-3}{k-3}+3.
$$
For $k=4$ or $k\ge 6$, the equality holds if and only if $\mathcal{F} $ is isomorphic to $ \mathcal{J}_3(n, k)$; 
For $k=5$, the equality holds 
if and only if  $\mathcal{F} $ is isomorphic to $\mathcal{J}_3(n, 5)$ or $\mathcal{G}_4(n, 5)$.
\end{theorem}

Apart from the above stability results, there are a large number of results involving the stability for $t$-intersecting families; see, e.g., \cite{BL2022,CLW2021,CLLW2022,GLX2022,L24,OV2021} and references therein. 
As mentioned before, various methods involving the 
celebrated Erd\H{o}s--Ko--Rado theorem are presented in the literature.  Motivated by these methods, 
there are also many alternative proofs for the Hilton--Milner theorem; see, e.g., \cite{BW2024,F19,FF1986,FT92, H18,KZ2018}. 
For the Han--Kohayakawa theorem and Huang--Peng theorem, a more general result of Kupavskii \cite{K19,Kup2024} can imply alternative proofs in the case $k\geq 5$. Unfortunately, the cases $k=3$ and $k=4$ cannot be proved following his method in a straightforward way. 
The arguments in these cases are relatively more complicated since some exceptional extremal families  appear. 
Recently, Ge, Xu and Zhao \cite{GXZ2024} 
provided another alternative proof of the Han--Kohayakawa theorem by using a robust linear algebra method, and they also obtained the $t$-th level stability result for the Erd\H{o}s--Ko--Rado theorem. 
It seems that the original proofs of Han and Kohayakawa \cite{H17} and Huang and Peng \cite{H24} are quite technical.  Therefore, it is  quite interesting to give new short proofs of the Han--Kohayakawa theorem as well as the Huang--Peng theorem. 

The main contribution of this paper is to present a unified simple approach to proving the stability results for intersecting families. 
For simplicity, we will not completely expand on these full arguments in the case $k=3$ or $4$.  For interested readers, we will provide detailed proofs in Appendices \ref{App-A} and \ref{App-B} using our unified framework. In addition, it seems feasible to further extend our result on cross-intersecting families to give a short proof of the stability result in \cite{GXZ2024}, we do not realize the details here, since the argument for small integer $k$ is a little bit more involved.


\section{Proof of Theorem \ref{51}}

\label{sec-2}

To prove Theorem \ref{51}, we review some fundamental notations and results about the shifting operation.
Let  $\mathcal{F} \subseteq 2^{[n]}$ be a family and $1 \leq i<j \leq n$. The \textit{shifting operator} $s_{i, j}$, discovered by Erdős, Ko and Rado \cite{E61}, is defined as 
$ s_{i, j}(\mathcal{F})=\left\{s_{i, j}(F): F \in \mathcal{F}\right\}$, where
$$
s_{i, j}(F)= \begin{cases}(F \backslash\{j\}) \cup\{i\} & \text { if } j \in F, i \notin F \text { and } (F \backslash\{j\}) \cup\{i\} \notin \mathcal{F}, \\ F & \text { otherwise. }\end{cases}
$$
Obviously, we have $\left|s_{i, j}(F)\right|=|F|$ and $\left|s_{i, j}(\mathcal{F})\right|=|\mathcal{F}|$.  A frequently used property  is that $s_{i, j}$ maintains the $t$-intersecting property of a family. 
A family $\mathcal{F} \subseteq 2^{[n]}$ is called \textit{shifted} if for all $ F \in \mathcal{F}$,  $i<j$ with $i\notin F$ and $j\in F$, then $(F \backslash\{j\}) \cup\{i\} \in \mathcal{F}$.  
It is well-known that every intersecting family can be transformed to a shifted intersecting family by applying shifting operations repeatedly.  
There are many nice properties  of  shifted families. For example, 
if $\mathcal{F}$ is a shifted family 
and $\{a_1,\ldots,a_k\}\in \mathcal{F}$ with $a_1<\cdots<a_k$, then for any set  $\{b_1,\ldots,b_k\}$ with $b_1<\cdots<b_k$ and $b_i\leq a_i$ for each $i\in[1,k]$, we have $\{b_1,\ldots,b_k\}\in \mathcal{F}$. 
For $\mathcal{F} \subseteq 2^{[n]}$ and $i\in [n]$, we denote 
$\mathcal{F}(\bar{i}) =\left\{F \in \mathcal{F}: i \notin F \right\}$ and 
$\mathcal{F}(i)=\left\{F\backslash \{i\}: i \in F \in \mathcal{F}\right\}$.

\begin{lemma}\label{S1}
Let $k\geq 1 $ and $ t\geq 0$ be integers. Let  $\mathcal{F} \subseteq 2^{[n]}$ be a shifted $(t+1)$-intersecting family. Then $\mathcal{F}(\bar{1})$ is $(t+2)$-intersecting. Moreover,  if $\mathcal{F} \subseteq\binom{[n]}{k+t}$ is a shifted $(t+1)$-intersecting family and $n\geq 2k+t$, then $\mathcal{F}(n)$ is $(t+1)$-intersecting.
\end{lemma}
\begin{proof}
 We may assume that $|\mathcal{F}(\bar{1})|\geq 2$ and  
$|\mathcal{F}(n)|\geq 2$.
 For any $F_1, F_2 \in \mathcal{F}(\bar{1})$, let $j\in F_1 \cap F_2$.
By shiftedness, we have $(F_1\backslash \{j\})\cup\{1\} \in \mathcal{F}$. It follows from the $(t+1)$-intersecting property of $\mathcal{F}$ that $| F_1 \cap F_2|=|\left((F_1\backslash \{j\})\cup\{1\} \right) \cap F_2|+1\geq t+2$.  For the second statement, we may assume that $\mathcal{F}(n)\neq\emptyset$. For any $E_1, E_2 \in \mathcal{F}(n)$, we have $E_1\cup\{n\}, E_2\cup\{n\}\in \mathcal{F}$. Observe that
$
|E_1\cup E_2\cup\{n\}|\leq 2(k+t)-(t+1)\leq n-1.
$
So there exists $x\notin E_1\cup E_2\cup\{n\}$ such that $E_1\cup\{x\} \in \mathcal{F}$. It is immediate that
$
|E_1\cap E_2|=|(E_1\cup\{x\})\cap (E_2\cup\{n\})|\geq t+1,
$
as desired.
\end{proof}

\begin{lemma}\label{S2}
Let $k\ge 1, t\ge 0 $ and $n\geq k+t+3$ be integers.
If  $\mathcal{F} \subseteq\binom{[n]}{k+t}$ is shifted and $|\mathcal{F}|\geq 3$, then $|\mathcal{F}(\bar{n})|\geq 3$.
\end{lemma}
\begin{proof}
If $\mathcal{F}(n)=\emptyset$, then $|\mathcal{F}(\bar{n})|=|\mathcal{F}|\geq 3$. If $\mathcal{F}(n)\neq\emptyset$, then there exists $F\in \mathcal{F}$ such that $n\in F$. Note that $n\geq k+t+3$. So there are at least three different elements $x, y, z\in [n]\backslash F $. By shiftedness, $(F\backslash\{n\})\cup\{x\}, (F\backslash\{n\})\cup\{y\}, (F\backslash\{n\})\cup\{z\}\in \mathcal{F}(\bar{n})$, and the result follows.
\end{proof}

\begin{lemma}[See \cite{E61}]  \label{S3}
 Let $\mathcal{F} \subseteq 2^{[n]}$ and $\mathcal{G} \subseteq 2^{[n]}$ be  cross-intersecting, and $\mathcal{F}$ be $t$-intersecting.
Then $s_{i, j}(\mathcal{F})$ and $s_{i, j}(\mathcal{G})$ are also cross-intersecting, and $s_{i, j}(\mathcal{F})$ is $t$-intersecting.
\end{lemma}

\begin{lemma}\label{S4}
Let $t\ge 0$, $k\geq 1$ and $n \geq 2 k+t$ be integers.
Let $\mathcal{F} \subseteq\binom{[n]}{k+t}$ and $\mathcal{G} \subseteq\binom{[n]}{k}$ be shifted cross-intersecting families. Then $\mathcal{F}(n)$ and $\mathcal{G}(n)$ are cross-intersecting.
\end{lemma}
\begin{proof}
The case for $\mathcal{F}(n)=\emptyset$ or $\mathcal{G}(n)=\emptyset$ is trivial. If $\mathcal{F}(n)\neq\emptyset$ and $\mathcal{G}(n)\neq\emptyset$, then there are $F_1\in \mathcal{F}(n)$ and $G_1\in \mathcal{G}(n)$ such that $F_1\cup\{n\}\in \mathcal{F}$ and $G_1\cup\{n\}\in \mathcal{G}$.  Observe that
$
|F_1\cup G_1\cup\{n\}|\leq 2k+t-1\leq n-1.
$
So there exists $x\notin F_1\cup G_1\cup\{n\}$ such that $F_1\cup\{x\} \in \mathcal{F}$. Then we have
$
|F_1\cap G_1|=|(F_1\cup\{x\})\cap (G_1\cup\{n\})|\geq 1,
$
as desired.
\end{proof}

We shall mention that the second statement in Lemma \ref{S1} and the result in Lemma \ref{S4}  appear in \cite{F16} as Claims $1$ and $2$, respectively.

Next we define the lexicographic order on the $k$-element subsets of $[n]$. We say that $F$ is smaller than $G$ in the lexicographic order if $\min (F \backslash G)<\min (G \backslash F)$ holds. For $0\leq m\leq \binom{n}{k}$, let $\mathcal{L}(n, k, m)$ be the family of the first $m$ $k$-sets in the lexicographic order. 
The following lemma will be used in the next section, for convenience, we list it below;  
see \cite[p.266]{FK2017} for a detailed proof.

\begin{lemma}[See \cite{H76}] \label{Ln}
Let $k, \ell, n$ be positive integers with $n> k+\ell$. If $\mathcal{F} \subseteq\binom{[n]}{k}$ and $\mathcal{G} \subseteq\binom{[n]}{\ell}$ are cross-intersecting, then $\mathcal{L}(n, k,|\mathcal{F}|)$ and $\mathcal{L}(n, \ell,|\mathcal{G}|)$ are cross-intersecting.
\end{lemma}

\noindent{\bf Proof of Theorem \ref{51}.}
For fixed $k$ and $t$, we apply induction on  $n\geq 2 k+t$.
First let us consider the base case $n=2 k+t$. For any $F \in\binom{[2k+t]}{k+t}$, the cross-intersecting property of $\mathcal{F}$ and $\mathcal{G}$ implies that $F \notin \mathcal{F}$ or $[2 k+t] \backslash F \notin \mathcal{G}$.
 Hence, at most half of the sets of $\binom{[2k+t]}{k+t}\cup \binom{[2k+t]}{k}$ belong to $\mathcal{F}\cup\mathcal{G}$.
Then
$|\mathcal{F}|+|\mathcal{G}| \leq\binom{2k+t}{k}$.
 Note that $\binom{2k+t}{k}-\binom{2k+t-k-t}{k}-\binom{2k+t-k-t-1}{k-1}-\binom{2k+t-k-t-2}{k-2}+3=\binom{2k+t}{k}$. Thus the result holds.

Now suppose that $n \geq 2 k+t+1$ and that the result holds for integers less than $n$ and fixed $k$ and $t$. By Lemma \ref{S3},  we can apply all shifting operations to $\mathcal{F}$ and $\mathcal{G}$, the resulting families $\mathcal{F}^{\prime}$ and $\mathcal{G}^{\prime}$ are shifted and cross-intersecting. Moreover,   $\mathcal{F}^{\prime}$ is $(t+1)$-intersecting.
Thus we may assume that $\mathcal{F}$ and $\mathcal{G}$ are shifted.
It follows from Lemma \ref{S2} that $|\mathcal{F}(\bar{n})| \geq 3$. Clearly, $\mathcal{F}(\bar{n})$ is $(t+1)$-intersecting. Note that $\mathcal{F}(\bar{n})$ and
$\mathcal{G}(\bar{n})$ are cross-intersecting.  By the  induction hypothesis, we get
$$
|\mathcal{F}(\bar{n})|+|\mathcal{G}(\bar{n})| \leq\binom{n-1}{k}-\binom{n-k-t-1}{k}-\binom{n-k-t-2}{k-1}-\binom{n-k-t-3}{k-2}+3.
$$

We  continue the proof by considering the following two cases.

{\bf Case 1.} Suppose that $|\mathcal{F}(n)|=0$.

 Since $|\mathcal{F}(\bar{n})| \geq 3$, there exist $F_1, F_2, F_3 \in \mathcal{F}(\bar{n})=\mathcal{F}$.
For any permutation $\sigma$ on $[3]$, we consider the following families:
\begin{align*}
\mathcal{G}_{\sigma(1)}&=\left\{G \in\binom{[n-1]}{k-1}: G \cap F_{\sigma(1)}=\emptyset\right\}, \\
 \mathcal{G}_{\sigma(2)} &=\left\{G \in\binom{[n-1]}{k-1}: G \cap F_{\sigma(1)} \neq \emptyset, G \cap F_{\sigma(2)}=\emptyset\right\},\\
 \mathcal{G}_{\sigma(3)}&=\left\{G \in\binom{[n-1]}{k-1}: G \cap F_{\sigma(1)} \neq \emptyset, G \cap F_{\sigma(2)}\neq\emptyset, G \cap F_{\sigma(3)}=\emptyset\right\}.
\end{align*}
 By the definition of $\mathcal{G}_{\sigma(1)}$, we have $|\mathcal{G}_{\sigma(1)}|=\binom{n-k-t-1}{k-1}$ as $F_{\sigma(1)}\in\binom{[n-1]}{k+t}$.
 Since $\mathcal{F}_{\sigma(1)}\neq \mathcal{F}_{\sigma(2)}$, there exists $x\in \mathcal{F}_{\sigma(1)}$ such that $x\notin \mathcal{F}_{\sigma(2)}$.
 Then $\left\{G \in\binom{[n-1]}{k-1}: x\in G, G \cap F_{\sigma(2)}=\emptyset\right\}
 \subseteq \mathcal{G}_{\sigma(2)}
 $
 and hence 
 $|\mathcal{G}_{\sigma(2)}|\geq\binom{n-k-t-2}{k-2}$. In addition, there exist $x_1\in \mathcal{F}_{\sigma(1)}$ and $ x_2\in \mathcal{F}_{\sigma(2)}$ such that $x_1, x_2\notin F_{\sigma(3)}$. If $x_1\neq x_2$, then $|\mathcal{G}_{\sigma(3)}|\geq\binom{n-k-t-3}{k-3}$; if $x_1= x_2$, then $|\mathcal{G}_{\sigma(3)}|\geq\binom{n-k-t-2}{k-2}$. So $|\mathcal{G}_{\sigma(3)}|\geq\binom{n-k-t-3}{k-3}$ holds in all cases.
Since $\mathcal{G}_{\sigma(i)}\cap \mathcal{G}_{{\sigma(j)}}=\emptyset$ for any $i\neq j$ and $\mathcal{F}=\mathcal{F}(\bar{n})$ and $\mathcal{G}(n)$ are cross-intersecting, we have
\begin{align*}
|\mathcal{G}(n)| \leq\binom{n-1}{k-1}-\sum_{j=1}^3|\mathcal{G}_{{\sigma(j)}}|\leq \binom{n-1}{k-1}-\binom{n-k-t-1}{k-1}-\binom{n-k-t-2}{k-2}-\binom{n-k-t-3}{k-3}.
\end{align*}
Next suppose that the equality holds. Then $\mathcal{F}(\bar{n})=\left\{F_1, F_2, F_3\right\}$ and for any permutation $\sigma$ on $[3]$, we have $|\mathcal{G}_{\sigma(2)}|=\binom{n-k-t-2}{k-2}$ and $|\mathcal{G}_{\sigma(3)}|=\binom{n-k-t-3}{k-3}$. If there exist $i\neq j\in [3]$ such that $\left|F_i \cap F_j\right|\leq k+t-2$, then let $\sigma=(1 i)(2 j)$. So there exist $y_1\neq y_2\in \mathcal{F}_{\sigma(1)}$ such that $y_1, y_2\notin F_{\sigma(2)}$. This implies that 
$\left\{G \in\binom{[n-1]}{k-1}: y_1\in G \text{ or } y_2\in G, G \cap F_{\sigma(2)}=\emptyset\right\}
 \subseteq \mathcal{G}_{\sigma(2)}
 $. Thus $|\mathcal{G}_{\sigma(2)}|\geq\binom{n-k-t-2}{k-2}+\binom{n-k-t-3}{k-2}$, a contradiction. Thus $\mathcal{F}(\bar{n})=\left\{F_1, F_2, F_3\right\}$ with $|F_{1}\cap F_{2}\cap F_{3}|=k+t-1$. Otherwise we can write $F_1=M\cup\{a, b\}, F_2=M\cup\{a, c\}, F_3=M\cup\{b, c\}$, where $|M|=k+t-2$ and $a, b, c$ are pairwise different.
 By symmetry, we may assume that $F_i=F_{\sigma(i)}$ for $i\in [3]$.
Then
$\left\{G \in\binom{[n-1]}{k-1}: a\in G, G \cap F_{\sigma(3)}=\emptyset\right\}
 \subseteq \mathcal{G}_{\sigma(3)}
 $.
Hence, for any permutation $\sigma$ on $[3]$, we have $|\mathcal{G}_{\sigma(3)}|\geq\binom{n-k-t-2}{k-2}$, a contradiction.
Therefore, by the shiftedness of $\mathcal{F}$, we conclude that the equality holds if and only if $\mathcal{F}(\bar{n})=\left\{F_1, F_2, F_3\right\}$ with $F_i:=[k+t-1]\cup \{k+t-1+i\}$ for each $i\in [3]$ and $\mathcal{G}(n)=\left\{G \in\binom{[n-1]}{k-1}: G \cap F_i \neq\emptyset \text{ for any } i=1, 2, 3\right\}$. 
It follows that
$$
|\mathcal{F}|+|\mathcal{G}|=|\mathcal{F}(\bar{n})|+|\mathcal{G}(\bar{n})|+|\mathcal{G}(n)|\leq \binom{n}{k} \!-\! 
\binom{n\!-\!k\!-\!t}{k} \!-\! 
\binom{n\!-\!k\!-\!t\!-\!1}{k-1} \!-\! 
\binom{n\!-\!k\!-\!t\!-\!2}{k-2}+3.
$$
The equality holds if and only if $\mathcal{F}=\left\{F_1, F_2, F_3\right\}$ with $F_i:=[k+t-1]\cup \{k+t-1+i\}$ for each $i\in [3]$ and $\mathcal{G}=\left\{G \in\binom{[n]}{k}: G \cap F_i \neq\emptyset \text{ for any } i\in[3]\right\}$.

{\bf Case 2.} Suppose that $|\mathcal{F}(n)|\geq1$.

By Lemma \ref{S4}, $\mathcal{F}(n)\subseteq\binom{[n-1]}{k+t-1}$ and $\mathcal{G}(n)\subseteq\binom{[n-1]}{k-1}$ are cross-intersecting. By Lemma \ref{S1}, $\mathcal{F}(n)$ is $(t+1)$-intersecting .
Since $n \geq 2 k+t+1$ and $k-1\geq 3$, applying Theorem \ref{F16} to $\mathcal{F}(n)$ and $\mathcal{G}(n)$ yields
$$
|\mathcal{F}(n)|+|\mathcal{G}(n)| \leq\binom{n-1}{k-1}-\binom{n-k-t}{k-1}+1.
$$
 Since $k-1\geq 3$, by Theorem \ref{F16}, the equality holds if and only if $\mathcal{F}(n)=\left\{F_1\right\}$ for some $F_1\in\binom{[n-1]}{k+t-1}$ and $\mathcal{G}=\left\{G \in\binom{[n-1]}{k-1}: G \cap F_1 \neq \emptyset \right\}$. Then the shiftedness of $\mathcal{F}$ leads to $\mathcal{F}(n)=\left\{[k+t-1]\right\}$ and $\mathcal{G}(n)=\left\{G^* \in\binom{[n-1]}{k-1}: G^* \cap [k+t-1] \neq \emptyset \right\}$. Note that $\mathcal{F}(\bar{n})$ and $\mathcal{G}(n)$ are cross-intersecting.
By the shiftedness of $\mathcal{F}$, if $\mathcal{F}(n)=\left\{[k+t-1]\right\}$ and $\mathcal{G}(n)=\left\{G^* \in\binom{[n-1]}{k-1}: G^* \cap [k+t-1] \neq \emptyset \right\}$, 
then $\mathcal{F}(\bar{n})=\left\{[k+t-1]\cup\{i\}:  i \in\{k+t, \ldots, n-1\}\right\}$, $\mathcal{G}(\bar{n})=\left\{G \in\binom{[n-1]}{k}: G \cap [k+t-1] \neq \emptyset \right\}$.

Therefore, we conclude that
\begin{align*}
|\mathcal{F}|+|\mathcal{G}|
&=|\mathcal{F}(\bar{n})|+|\mathcal{G}(\bar{n})|+|\mathcal{F}(n)|+|\mathcal{G}(n)| \\
& \leq\binom{n-1}{k}-\binom{n-k-t-1}{k}-\binom{n-k-t-2}{k-1}
-\binom{n-k-t-3}{k-2}+3 \\ 
& \quad +\binom{n-1}{k-1}-\binom{n-k-t}{k-1}+1\\
&=\binom{n}{k} \!-\! \binom{n-k-t}{k}\!-\!\binom{n-k-t-1}{k-1} 
\!-\! \binom{n-k-t-2}{k-2}+3+1\!-\! \binom{n-k-t-3}{k-4}\\
&\leq \binom{n}{k}-\binom{n-k-t}{k}-\binom{n-k-t-1}{k-1}-\binom{n-k-t-2}{k-2}+3,
\end{align*}
 where the second equality holds by applying the Pascal identity five times.
Moreover, if all the  equalities above  hold, then $1=\binom{n-k-t-3}{k-4}$, i.e., $k=4$, and $\mathcal{F}=\{[t+3] \cup\{i\}: i \in\{t+4, \ldots, n\}\}$ and $\mathcal{G}=\left\{G \in\binom{[n]}{4}: G \cap[t+3] \neq \emptyset\right\}$.

In the above discussion, we have determined the extremal families 
$\mathcal{F}$ and $\mathcal{G}$ that attain the required upper bound  under the shifting assumption by Lemma \ref{S3}. 
For completeness, we also need to characterize the extremal families in general case. 
Assume that  the extremal families $\mathcal{F}^{\prime}$ and $\mathcal{G}^{\prime}$ satisfy $s_{i, j}\left(\mathcal{F}^{\prime}\right)=\mathcal{F}$ and $s_{i,j}\left(\mathcal{G}^{\prime}\right)=\mathcal{G}$.
Note that $i<j$.
When $\mathcal{F}=\left\{F_1,F_2,F_3 \right\}$ with 
$F_i:=[k+t-1]\cup \{k+t-1+i\}$ for each $i\in [3]$, if $i,j\in [k+t-1]$, then $\mathcal{F}^{\prime}=\mathcal{F}$.
If $i\in [k+t-1]$ and $j\in [k+t, k+t+2]$, by symmetry, we assume that $i=1$ and $j=k+t$, then $\mathcal{F}^{\prime}=\{[2,k+t\}\cup\{q\}: q\in\{1,k+t+1, k+t+2\}\}$. Thus $\mathcal{F}^{\prime} \cong\mathcal{F}$. If $i\in [k+t-1]$ and $j\notin [1, k+t+2]$, we may assume that $i=1$, then 
$\mathcal{F}^{\prime}=\{[2,k+t-1\}\cup\{q,j\}: q\in[k+t, k+t+2]\}$. Thus  $\mathcal{F}^{\prime}\cong\mathcal{F}$. 
If $i,j\in[k+t, k+t+2]$, then $\mathcal{F}^{\prime}=\mathcal{F}$.
If $i\in[k+t, k+t+2]$ and $j\notin [1, k+t+2]$, we may assume that $i=k+t$ and $j=k+t+3$, then $\mathcal{F}^{\prime}=\{[1,k+t-1\}\cup\{q\}: q\in[k+t+1, k+t+3]\}$.  Thus  $\mathcal{F}^{\prime}\cong\mathcal{F}$.
If $i,j\notin [1, k+t+2]$, then $\mathcal{F}^{\prime} =\mathcal{F}$. 
Therefore, we conclude that in the case $\mathcal{F}=\left\{F_1,F_2,F_3 \right\}$ with 
$F_i:=[k+t-1]\cup \{k+t-1+i\}$ for each $i\in [3]$, we have $\mathcal{F}^{\prime} \cong\mathcal{F}$. The family $\mathcal{G}^{\prime}$ is just the maximal subfamily of $\binom{[n]}{k}$ that is cross-intersecting with $\mathcal{F}^{\prime}$.
The same is true when $\mathcal{F}$ is the other case.
Let $\mathcal{F}_0$ and $\mathcal{G}_0$ be the original families before applying the shifting operations. 
In other words, $\mathcal{F}$ and $ \mathcal{G}$ are obtained from $\mathcal{F}_0 $ and $ \mathcal{G}_0$ by applying a series of shifting operations.
Then we have $\mathcal{F}_0=\mathcal{F}$ and $\mathcal{G}_0=\mathcal{G}$, under isomorphism. 
$\hfill \square$

\section{Revisiting stability of intersecting families}

\label{sec-3}

As promised, we will present a unified framework to 
revisit several stability results for the Erd\H{o}s--Ko--Rado theorem.  
To begin with, we fix the following notation. 
\begin{align*}
\mathcal{F}(i,j)&=\left\{F\backslash \{i,j\}: i, j\in F \in \mathcal{F}\right\},\\
\mathcal{F}(i,\bar{j})&=\mathcal{F}(\bar{j},i)=\left\{F\backslash \{i\}: i \in F, j\notin F, F \in \mathcal{F}\right\},\\ \mathcal{F}(\bar{i},\bar{j})&=\left\{F: i, j\notin F \in \mathcal{F}\right\}.
\end{align*}

\subsection{Short proof of the Hilton--Milner theorem}

To more clearly illustrate our approach, 
we first give a short new proof of Hilton--Milner's bound in Theorem \ref{H67}. 
It is worth mentioning that Frankl and Tokushige \cite{FT92} provided a short proof of the bound in Theorem \ref{H67} by using the Katona theorem on shadows for intersecting families \cite{K64}. Different from their approach, 
we here present a short proof of Theorem \ref{H67} by using Theorem \ref{F16} only,  avoiding the use of the Katona intersecting shadow theorem. 

\vspace{3mm}
\noindent{\bf Proof of Theorem \ref{H67}.}
Suppose that  $\mathcal{F}$ is not EKR and $|\mathcal{F}|$ is maximal. As is well-known, the shifting operation preserves the intersecting property. 
Applying shifting operations to $\mathcal{F}$ repeatedly, 
we can arrive at the following two cases: 
\begin{itemize} 
\item[(A)]
After applying all shifting operations, 
we get a shifted family that is not EKR. 

\item[(B)] 
After some shifting operations, 
 at some point, the resulting family is EKR. 
 \end{itemize}

{\bf Case (A).} Suppose that $\mathcal{H}$  is obtained from $\mathcal{F}$ by applying all shifts and  $\mathcal{H}$ is not EKR. 
Then we obtain $\mathcal{H}(\bar{1})\neq \emptyset$.
By Lemma \ref{S1}, $\mathcal{H}(\bar{1})$ is 2-intersecting. Note that $\mathcal{H}(\bar{1}) \subseteq\binom{[2,n]}{k}$ and $\mathcal{H}(1) \subseteq\binom{[2,n]}{k-1}$ are  cross-intersecting. Applying Theorem \ref{F16} (by setting $t=1$), we get
\begin{equation*} \label{eq-HM}
 |\mathcal{F}| = 
|\mathcal{H}|=|\mathcal{H}(\bar{1})|+|\mathcal{H}(1)| \leq\binom{n-1}{k-1}-\binom{n-k-1}{k-1}+1.
\end{equation*}
Moreover, the above equality holds if and only if 
$\mathcal{H}(\bar{1})=\{[2,k+1]\}$ and $\mathcal{H}(1)=\{H\in {[2,n] \choose k-1}:H\cap [2,k+1]\neq \emptyset \}$, 
or in the case $k=3$, there is one more possibility: $\mathcal{H}(\bar{1})=\big\{ \{2,3\}\cup \{i\}: 4\le i\le n \big\}$ and 
$\mathcal{H}(1)=\{H\in {[2,n] \choose 2}: H\cap \{2,3\}\neq \emptyset\}$. Consequently, we get $\mathcal{H}=\mathcal{HM}(n, k)$, or $\mathcal{H}=\mathcal{T}(n, 3)$. 
It is easy to check that if $\mathcal{H}^{\prime}$ is an extremal family and $s_{i,j}(\mathcal{H}^{\prime})=\mathcal{H}$, then $\mathcal{H}^{\prime}\cong \mathcal{H}$.

{\bf Case (B).}
Let $\mathcal{G} \subseteq\binom{[n]}{k}$ be obtained from $\mathcal{F}$ by applying some shifting operations such that $\mathcal{G}$  is not EKR,  but after the shifting operation $s_{i,j}$, the resulting family $s_{i,j}(\mathcal{G})$ is EKR. 
In this case, all sets of $s_{i, j}(\mathcal{G})$  contain the element $i$, 
it follows that $\mathcal{G}(\bar{i},\bar{j})= \emptyset$ and  hence   $\{i,j\}\cap G\neq \emptyset$ for every $G\in \mathcal{G}$.  
Moreover, $\mathcal{G}(\bar{i},j)\cap \mathcal{G}(i,\bar{j})=\emptyset$ holds.
 By the maximality of $|\mathcal{G}|$, we have
$\mathcal{G}(i,j)=\binom{[n]\backslash \{i,j\}}{k-2}$.  Since $\mathcal{G}$ is not EKR, we get that
$\mathcal{G}(i,\bar{j})$ and $\mathcal{G}(\bar{i},j)$ are non-empty.
Note that $\mathcal{G}(i,\bar{j})\subseteq\binom{[n]\backslash \{i,j\}}{k-1}$ and $\mathcal{G}(\bar{i},j)\subseteq\binom{[n]\backslash \{i,j\}}{k-1}$ are cross-intersecting.
So Theorem \ref{mainH} can be applied. 
Since $\mathcal{G}(\bar{i},j)\cap \mathcal{G}(i,\bar{j})=\emptyset$,  the inequality in Theorem \ref{mainH} holds strictly. Then 
\begin{align*}
|\mathcal{F}|=|\mathcal{G}|&=|\mathcal{G}(i,j)|+|\mathcal{G}(i,\bar{j})|+|\mathcal{G}(\bar{i}, j)| \\ 
&< \binom{n-2}{k-2}+\binom{n-2}{k-1}-\binom{n-k-1}{k-1} + 1\\
&=\binom{n-1}{k-1}-\binom{n-k-1}{k-1}+1,
\end{align*}
as desired. This completes the proof of Theorem \ref{H67}.$\hfill \square$

\medskip 
\noindent
{\bf Remark.} After submitting this paper, we are informed by Bulavka and Woodroofe \cite{BW2024} that their recent work on cross-intersecting families also yields a short proof of the Hilton--Milner theorem. We remark here that the main result  \cite[Theorem 2]{BW2024} is a direct consequence of Theorem \ref{F16}. To some extent, our framework used in the above proof surprisingly develops a unified approach for deriving deep stability results for intersecting families, not just the Hilton--Milner theorem. For instance, we will present more potential applications of our approach  in next two subsections.

\subsection{Short proof of the Han--Kohayakawa theorem}
To make the proof clearer, we divide it into case $k\geq 4$ and  case $k=3$ (See Appendix  \ref{App-A}). Here we consider the case $k\geq 4$ and present some tools.   Note that the case when $k=3$ is in the appendix by something along the lines of this subsection.

For two integers $k, \ell$ with $n \geq k+\ell$ and a family $\mathcal{F} \subseteq\binom{[n]}{k}$, we define
$$
\mathcal{D}_{\ell}(\mathcal{F})=\left\{D \in\binom{[n]}{\ell}: \exists F \in \mathcal{F}, D \cap F=\emptyset\right\}.
$$
Then  $\mathcal{F}$ and $\mathcal{G} \subseteq\binom{[n]}{\ell}$  are cross-intersecting if and only if $\mathcal{F} \cap \mathcal{D}_k(\mathcal{G})=\emptyset$ or equivalently $\mathcal{G} \cap \mathcal{D}_{\ell}(\mathcal{F})=\emptyset$. Moreover, for given $\mathcal{F}$, $\mathcal{G}=\binom{[n]}{\ell} \backslash \mathcal{D}_{\ell}(\mathcal{F})$ is the largest $\ell$-uniform  family that is cross-intersecting with $\mathcal{F}$, and for given $\mathcal{G}$, $\mathcal{F}=\binom{[n]}{k} \backslash \mathcal{D}_{k}(\mathcal{G})$ is the largest  $k$-uniform   family that is cross-intersecting with $\mathcal{G}$.

We need the following lemma to characterize the uniqueness of extremal families.

\begin{lemma}[See \cite{F86, M85}] \label{FM}
Let $n> k+\ell, \mathcal{F} \subseteq\binom{[n]}{k}$ and 
$ |\mathcal{F}|=\binom{n-r}{k-r}$ for some $r \leq k$. Then
$
\left|\mathcal{D}_{\ell}(\mathcal{F})\right| \geq\binom{ n-r}{\ell},
$
with strict inequality unless $\mathcal{F}=\big\{F \in\binom{[n]}{k}: R \subseteq F\big\}$ for some $R \in \binom{[n]}{r}$.
\end{lemma}

The following lemma is a direct consequence of Theorem \ref{W23}. 

\begin{lemma}\label{W231}
Let $k \geq 3$ and $n \geq 2 k$ be positive integers.  Let $\mathcal{F} \subseteq\binom{[n]}{k}$ and $\mathcal{G} \subseteq\binom{[n]}{k}$ be  cross-intersecting families. Suppose that $|\mathcal{G}| \geq|\mathcal{F}| \geq 2$, then
$$
|\mathcal{F}|+|\mathcal{G}| \leq\binom{n}{k}-\binom{n-k}{k}-\binom{n-k-1}{k-1}+2.
$$
For $n > 2 k$,  the equality holds if and only if $\mathcal{F}=\left\{F_1, F_2\right\}$ for some $F_1, F_2 \in\binom{[n]}{k}$ with $\left|F_1 \cap F_2\right|=k-1$ and $\mathcal{G}=\left\{G \in\binom{[n]}{k}: G \cap F_1 \neq \emptyset \text{ and } G \cap F_2 \neq \emptyset\right\}$, or two more possibilities when $k=3$, namely, $\mathcal{F}=\{[2] \cup\{i\}: i \in\{3, 4, \ldots, n\}\}\text{ and } \mathcal{G}=\left\{G \in\binom{[n]}{3}: G \cap[2] \neq \emptyset\right\}$, or $\mathcal{F}=\mathcal{G}=\left\{G \in\binom{[n]}{3}: 1\in G \right\}$, under isomorphism.
\end{lemma}
\begin{proof}
If $n = 2 k$, then $
\binom{n}{k}-\binom{n-k}{k}-\binom{n-k-1}{k-1}+2=\binom{2k}{k}$.
Hence the upper bound for $n = 2 k$ is trivial. Suppose that $n > 2 k$. We denote $\mathcal{F}':=\mathcal{L}(n, k,|\mathcal{F}|)$ and $\mathcal{G}':=\mathcal{L}(n, k,|\mathcal{G}|)$. 
By Lemma \ref{Ln}, we know that 
$\mathcal{F}'$ and $\mathcal{G}'$ are cross-intersecting. Since $|\mathcal{G}'| \geq|\mathcal{F}'| \geq r$, we 
have $\mathcal{F}' \subseteq \mathcal{G}'$, which implies that $\mathcal{F}'$ is intersecting. Then setting $t=0$ in Theorem \ref{W23} yields $$
|\mathcal{F}|+|\mathcal{G}| = 
    |\mathcal{F}'| + |\mathcal{G}'| \leq\binom{n}{k}-\binom{n-k}{k}-\binom{n-k-1}{k-1}+2.
$$
Moreover, for $n > 2 k$, the equality holds if and only if $\mathcal{F}'=\left\{[k]\cup\{i\}: i\in\{k,k+1\}\right\}$ and $\mathcal{G}'=\left\{G \in\binom{[n]}{k}: G \cap ([k]\cup\{i\})\neq \emptyset, i\in\{k,k+1\}\right\}$, or two more possibilities when $k=3$, namely, $\mathcal{F}'=\{[2] \cup\{i\}: i \in\{3, 4, \ldots, n\}\}\text{ and } \mathcal{G}'=\left\{G \in\binom{[n]}{3}: G \cap[2] \neq \emptyset\right\}$, or $\mathcal{F}'=\mathcal{G}'=\left\{G \in\binom{[n]}{3}: 1\in G \right\}$.
We need to return to the structure of $\mathcal{F}$ and $\mathcal{G}$. 
In the first extremal case, 
we have $|\mathcal{F}|=|\mathcal{F}'|=2$. 
Denote $\mathcal{F}=\left\{F_1, F_2\right\}$. Then $\left|F_1 \cap F_2\right|=k-1$  by the case of equality in Theorem \ref{W23}. Thus $\mathcal{G}=\left\{G \in\binom{[n]}{k}: G \cap F_i \neq\emptyset, i\in[2]\right\}$. 
In the second extremal case, we have $|\mathcal{F}|=|\mathcal{F}'|=n-2 = 
{n-2 \choose 3-2}$ and $|\mathcal{G}|=|\mathcal{G}'|={n \choose3} - {n-2 \choose 3}$. 
For a fixed family $\mathcal{F}\subseteq \binom{[n]}{3}$, 
the maximality of $|\mathcal{F}| + |\mathcal{G}|$ yields $\mathcal{G} = \binom{[n]}{3} \setminus \mathcal{D}_3(\mathcal{F})$. So we get  
$|\mathcal{D}_3(\mathcal{F})|={n-2 \choose 3}$, which implies that the equality case of  Lemma \ref{FM} occurs. 
Thus, it follows from Lemma \ref{FM} that  $\mathcal{F}=\{[2] \cup\{i\}: i \in\{3,4,  \ldots, n\}\}$ and  $\mathcal{G}= \left\{G \in\binom{[n]}{3}: G \cap[2] \neq \emptyset\right\}$ under isomorphism.
In the third extremal case, we have 
$|\mathcal{F}|=|\mathcal{G}|={n-1 \choose 2}$. Similarly, we get  $| \mathcal{D}_{3}(\mathcal{F})|=\binom{n-1}{3}$. By Lemma \ref{FM}, we obtain $\mathcal{F}= \left\{F\in\binom{[n]}{3}: 1\in F\right\}$ and 
$\mathcal{G}= \left\{G\in\binom{[n]}{3}: 1\in G\right\}$ under isomorphism.  
\end{proof}

\noindent{\bf Proof of Theorem \ref{H17} for $k\geq 4$.}
Suppose that $k\geq 4$, $\mathcal{F}\subseteq\binom{[n]}{k}$ is neither EKR nor HM, and $|\mathcal{F}|$ is maximal. We need to consider the following three cases.

{\bf Case 1.} After applying all possible shifting operations to $\mathcal{F}$ repeatedly, the resulting family is still neither EKR nor HM. 

Note that the resulting family is shifted and has the same size as $\mathcal{F}$. 
In this case, we may assume that $\mathcal{F}$ is shifted. Since $\mathcal{F}$ is neither EKR nor HM, we have $\mathcal{F}(\bar{1})\neq \emptyset$ and $|\mathcal{F}(\bar{1})|\geq 2$.
By Lemma \ref{S1}, $\mathcal{F}(\bar{1})$ is 2-intersecting. Note that $\mathcal{F}(\bar{1}) \subseteq\binom{[n]\backslash \{1\}}{k}$ and $\mathcal{F}(1) \subseteq\binom{[n]\backslash \{1\}}{k-1}$ are  cross-intersecting. Since $k-1\geq 3$, applying Theorem \ref{W23} (for $t=1$) yields
$$
|\mathcal{F}|=|\mathcal{F}(\bar{1})|+|\mathcal{F}(1)| \leq\binom{n-1}{k-1}-\binom{n-k-1}{k-1}-\binom{n-k-2}{k-2}+2.
$$
According to the extremal families in Theorem \ref{W23}, we have
$\mathcal{F}\cong\mathcal{J}_2(n, k)$ or two more possibilities when $k=4$,  $\mathcal{F}\cong\mathcal{G}_3(n, 4)$ or $\mathcal{F}\cong\mathcal{G}_2(n, 4)$.

{\bf Case 2.}
 Suppose that $\mathcal{F}$ is not shifted, and there exists $\mathcal{G} \subseteq\binom{[n]}{k}$ obtained from $\mathcal{F}$ by repeated some shifting operations such that $\mathcal{G}$  is neither EKR nor HM, $|\mathcal{G}|=|\mathcal{F}|$, but
$s_{i, j}(\mathcal{G})$ is EKR.

It is clear that $\cap_{H\in s_{i, j}(\mathcal{G})}H=\{i\}$ and $\{i,j\}\cap G\neq \emptyset$ for all $G\in \mathcal{G}$. Then $\mathcal{G}(\bar{i},\bar{j})= \emptyset$. It follows that $|\mathcal{G}(i,\bar{j})|\geq 2$ and $|\mathcal{G}(\bar{i},j)|\geq 2$. In addition, we have  $\mathcal{G}(\bar{i},j)\cap \mathcal{G}(i,\bar{j})=\emptyset$.
By the maximality of $|\mathcal{G}|$, we have
$\mathcal{G}(i,j)=\binom{[n]\backslash \{i,j\}}{k-2}$.  Note that $\mathcal{G}(i,\bar{j})\subseteq\binom{[n]\backslash \{i,j\}}{k-1}$ and $\mathcal{G}(\bar{i},j)\subseteq\binom{[n]\backslash \{i,j\}}{k-1}$ are cross-intersecting, and $k-1\geq 3$. Note that the inequality in Lemma \ref{W231} holds strictly.
Then we get
\begin{equation*}
\begin{aligned}
|\mathcal{F}|&=|\mathcal{G}(i,j)|+|\mathcal{G}(i,\bar{j})|+|\mathcal{G}(\bar{i}, j)| \\ 
& < \binom{n-2}{k-2}+\binom{n-2}{k-1}-\binom{n-k-1}{k-1}-\binom{n-k-2}{k-2}+2\\
&= \binom{n-1}{k-1}-\binom{n-k-1}{k-1}-\binom{n-k-2}{k-2}+2.
\end{aligned}
\end{equation*}

{\bf Case 3.}
Suppose that $\mathcal{F}$ is not shifted, and there exists $\mathcal{G} \subseteq\binom{[n]}{k}$ obtained from $\mathcal{F}$ by repeated shifting operations such that $\mathcal{G}$  is neither EKR nor HM, $|\mathcal{G}|=|\mathcal{F}|$, but $s_{i, j}(\mathcal{G})$ is HM.

We first have that $s_{i, j}(\mathcal{G})$ is HM at $i$.  Then $s_{i,j}(\mathcal{G})\subseteq \big\{F \in\binom{[n]}{k}: i \in F, F \cap B \neq \emptyset\big\} \cup\{B\}$ for some $i\notin B$, 
This implies that
 $\mathcal{G}(\bar{i}, \bar{j})=s_{i,j}(\mathcal{G}(\bar{i}, \bar{j}))\subseteq \{B\}$ and hence $|\mathcal{G}(\bar{i},\bar{j})|\leq 1$. By Case 2, we may assume that $s_{i, j}(\mathcal{G})$ is not EKR. If $|\mathcal{G}(\bar{i},\bar{j})|=0$, then $|\mathcal{G}(i,\bar{j})|\geq 2$ and $|\mathcal{G}(\bar{i},j)|\geq 2$. In addition, we have  $|\mathcal{G}(\bar{i},j)\cap \mathcal{G}(i,\bar{j})|=1$. Observe that the inequality in Lemma \ref{W231} still holds strictly for  $\mathcal{G}(i,\bar{j})$ and $\mathcal{G}(\bar{i},j)$,  and  $|\mathcal{G}(i,j)|\leq \binom{n-2}{k-2}$. So the same argument with Case 2 works. 
Now assume that $|\mathcal{G}(\bar{i},\bar{j})|= 1$.
 Since $\mathcal{G}(i,j)$ and $\mathcal{G}(\bar{i},\bar{j})$ are cross-intersecting, we have
$
|\mathcal{G}(i,j)|\leq \binom{n-2}{k-2}-\binom{n-k-2}{k-2}.
$
In addition, by the properties of $\mathcal{G}$, $\mathcal{G}(i,\bar{j})$ and $\mathcal{G}(\bar{i},j)$ are cross-intersecting, $\mathcal{G}(\bar{i},j)\cap \mathcal{G}(i,\bar{j})=\emptyset$, $|\mathcal{G}(i,\bar{j})|\geq 1$ and $|\mathcal{G}(\bar{i},j)|\geq 1$.   The inequality in 
Theorem \ref{mainH}  holds strictly. Then 
$
|\mathcal{G}(i,\bar{j})|+|\mathcal{G}(\bar{i}, j)|\leq \binom{n-2}{k-1}-\binom{n-k-1}{k-1}$.
Consequently,
\begin{align*}
|\mathcal{F}|&=|\mathcal{G}(i,j)|+|\mathcal{G}(i,\bar{j})|+|\mathcal{G}(\bar{i}, j)|+1 \\ 
& \leq \binom{n-2}{k-2}-\binom{n-k-2}{k-2}+\binom{n-2}{k-1}-\binom{n-k-1}{k-1}+1\\
&<\binom{n-1}{k-1}-\binom{n-k-1}{k-1}-\binom{n-k-2}{k-2}+2.
\end{align*}
By Case 1-Case 3, we complete the proof of Theorem \ref{H17} for $k\geq 4$.$\hfill \square$\vspace{3mm}

In order to prove Theorem \ref{H17} for $k=3$, we need to establish the corresponding versions of $k=2$ for Theorem \ref{W23} and Lemma \ref{W231}. 
For simplicity, we postpone the proof in Appendix \ref{App-A}.

\subsection{Short proof of the Huang--Peng theorem}
In Theorem \ref{H24}, suppose that there exists $x \in[n]$ such that there are only two sets $E_1$ and $E_2 \in \mathcal{F}$ missing $x$. If $\left|E_1 \cap E_2\right|=k-i$ and $i \geq 2$, then
\begin{align*}
|\mathcal{F}| & \leq\binom{n-1}{k-1}-2\binom{n-k-1}{k-1}+\binom{n-k-i-1}{k-1}+2 \\
& \leq\binom{n-1}{k-1}-2\binom{n-k-1}{k-1}+\binom{n-k-3}{k-1}+2.
\end{align*}
The equality holds if and only if $\left|E_1 \cap E_2\right|=k-2$, i.e., $\mathcal{F}=\mathcal{K}_2(n, k)$.
Moreover, it is not hard to see that
$$
|\mathcal{J}_3(n, k)| =\binom{n-1}{k-1}-\binom{n-k-1}{k-1}-\binom{n-k-2}{k-2}-\binom{n-k-3}{k-3}+3.
$$
 Then $\left|\mathcal{K}_2(n, k)\right| \geq\left|\mathcal{J}_3(n, k)\right|$ if and only if
$\binom{n-1}{k-1}-2\binom{n-k-1}{k-1}+\binom{n-k-3}{k-1}+2\geq \binom{n-1}{k-1}-\binom{n-k-1}{k-1}-\binom{n-k-2}{k-2}-\binom{n-k-3}{k-3}+3$, and if and only if
$\binom{n-k-3}{k-3}\geq \binom{n-k-3}{k-2}+1$, and if and only if  $2 k+1 \leq n \leq 3 k-3$. Moreover,  the equality holds only for $k=4$ and $n=9$.
In addition, for $n \geq 3 k-2$, $\left|\mathcal{K}_2(n, k)\right|<\left|\mathcal{J}_3(n, k)\right|$. Therefore, in the following, we may assume that for any $x \in[n]$, there are at least three sets in $\mathcal{F}$ not containing $x$. Then it suffices to show that
$$
|\mathcal{F}| \leq\binom{n-1}{k-1}-\binom{n-k-1}{k-1}-\binom{n-k-2}{k-2}-\binom{n-k-3}{k-3}+3.
$$
For $k=5$, the equality holds only for $\mathcal{F}=\mathcal{J}_3(n, k)$ or $\mathcal{G}_4(n, k)$. For every other $k$, the equality holds only for $\mathcal{F}=\mathcal{J}_3(n, k)$.
We divide the proof into case $k\geq 5$ and  case $k=4$. For $k\geq 5$, we need the following tools.

\medskip 

Similar to Lemma \ref{W231}, we can translate Theorem \ref{51} into the following lemma.

\begin{lemma}\label{W232}
Let $k \geq 4$ and $n \geq 2 k$ be positive integers.  Let $\mathcal{F} \subseteq\binom{[n]}{k}$ and $\mathcal{G} \subseteq\binom{[n]}{k}$ be  cross-intersecting families. Suppose that $|\mathcal{G}| \geq|\mathcal{F}| \geq 3$. Then
$$
|\mathcal{F}|+|\mathcal{G}| \leq\binom{n}{k}-\binom{n-k}{k}-\binom{n-k-1}{k-1}-\binom{n-k-2}{k-2}+3.
$$
For $n >2k$, the equality holds if and only if $\mathcal{F}=\left\{F_1, F_2, F_3\right\}$ with $F_i:=[k-1]\cup \{k-1+i\}$ for each $i\in [3]$  and $\mathcal{G}=\left\{G \in\binom{[n]}{k}: G \cap F_i \neq\emptyset \text{ for any } i=1, 2, 3\right\}$, or one more possibility when $k=4, \mathcal{F}=\{[3] \cup\{i\}: i \in\{4, \ldots, n\}\}$ and $\mathcal{G}=\left\{G \in\binom{[n]}{4}: G \cap[3] \neq \emptyset\right\}$, under isomorphism.
\end{lemma}

\noindent{\bf Proof of Theorem \ref{H24} for $k\geq 5$.}
Let $\mathcal{F} \subseteq\binom{[n]}{k}$ be an maximum intersecting family which is neither EKR nor HM, and $\mathcal{F} \nsubseteq \mathcal{J}_2(n, k)$. For any $x \in[n]$, let $|\mathcal{F}(\bar{x})|\geq 3$. Furthermore, 
for any $\mathcal{G}$ obtained from $\mathcal{F}$ by repeated shifting operations such that $\mathcal{G}$  is neither EKR nor HM, and  $\mathcal{G} \nsubseteq \mathcal{J}_2(n, k)$, we may assume that $|\mathcal{G}(\bar{x})|\geq 3$. We need to deal with the following four cases.

{\bf Case 1.} Applying all shifting operations to $\mathcal{F}$ repeatedly, the resulting shifted family is still not EKR, not HM and not a subfamily of $\mathcal{J}_2(n, k)$. 

Without loss of generality, we may assume that $\mathcal{F}$ is shifted. 
 First note that $|\mathcal{F}(\bar{1})|\geq 3$ and by Lemma \ref{S1}, $\mathcal{F}(\bar{1})$ is 2-intersecting. In addition, $\mathcal{F}(\bar{1}) \subseteq\binom{[n]\backslash \{1\}}{k}$ and $\mathcal{F}(1) \subseteq\binom{[n]\backslash \{1\}}{k-1}$ are  cross-intersecting. Since $k-1\geq 4$, Theorem \ref{51} ($t=1$) leads to
$$
|\mathcal{F}|=|\mathcal{F}(\bar{1})|+|\mathcal{F}(1)| \leq\binom{n-1}{k-1}-\binom{n-k-1}{k-1}-\binom{n-k-2}{k-2}-\binom{n-k-3}{k-3}+3.
$$
Moreover, for \(n \geq 2k + 1\), the equality holds if and only if
\begin{align*}
\mathcal{F}(\overline{1}) &= \left\{ \{2, \ldots, k, k+1\}, \{2, \ldots, k, k+2\}, \{2, \ldots, k, k+3\} \right\},\\
\mathcal{F}(1) &= \left\{ G \in \binom{[n] \setminus \{1\}}{k-1} : G \cap ([2, k] \cup \{i\}) \neq \emptyset \text{ for } i = k+1, k+2, k+3 \right\},
\end{align*}
or one more possibility when \(k - 1 = 4\), i.e., \(k = 5\),
\begin{align*}
\mathcal{F}(\overline{1}) = \left\{ [2, 5] \cup \{i\} : i \in \{6, \ldots, n\} \right\},\ \mathcal{F}(1) = \left\{ G \in \binom{[n] \setminus \{1\}}{4} : G \cap [2, 5] \neq \emptyset \right\}.
\end{align*}
It follows that
\begin{align*}
\mathcal{F} = &\left\{ G \in \binom{[n]}{k} : 1 \in G, G \cap ([2, k] \cup \{i\}) \neq \emptyset \text{ for } i = k+1, k+2, k+3 \right\} \\
& \cup \{2, \ldots, k, k+1\} \cup \{2, \ldots, k, k+2\} \cup \{2, \ldots, k, k+3\},
\end{align*}
or one more possibility when \(k = 5\),
$$
\mathcal{F} = \left\{ G \in \binom{[n]}{5} : 1 \notin G, [2, 5] \subseteq G \right\} \cup \left\{ G \in \binom{[n]}{5} : 1 \in G, G \cap [2, 5] \neq \emptyset \right\}.
$$
For the former, let \(x_0 = 1\), \(E = [2, k]\) and \(J_3 = \{k+1, k+2, k+3\}\). Then \(\mathcal{F} = \mathcal{J}_3(n, k)\). For the latter, let \(x_0 = 1\) and \(E = [2, 5]\). Then \(\mathcal{F} = \mathcal{G}_4(n, 5)\). Moreover, for any \(\mathcal{F}' \subseteq \binom{[n]}{k}\) with \(s_{ij}(\mathcal{F}') = \mathcal{F}\), we have \(\mathcal{F}' \cong \mathcal{J}_3(n, k)\) or one more possibility when \(k = 5\), \(\mathcal{F}' \cong \mathcal{G}_4(n, 5)\).

{\bf Case 2.}
 Suppose that $\mathcal{F}$ is not shifted, and there exists $\mathcal{G} \subseteq\binom{[n]}{k}$ obtained from $\mathcal{F}$ by repeated shifting operations such that $\mathcal{G}$  is neither EKR nor HM, and  $\mathcal{G} \nsubseteq \mathcal{J}_2(n, k)$, $|\mathcal{G}|=|\mathcal{F}|$, but
$s_{i, j}(\mathcal{G})$ is EKR.

It is clear that $\cap_{H\in s_{i, j}(\mathcal{G})}H=\{i\}$ and $\{i,j\}\cap G\neq \emptyset$ for all $G\in \mathcal{G}$. Hence, $\mathcal{G}(\bar{i},\bar{j})= \emptyset$. This implies that $|\mathcal{G}(i,\bar{j})|=|\mathcal{G}(\bar{j})|\geq 3$ and $|\mathcal{G}(\bar{i},j)|=|\mathcal{G}(\bar{i})|\geq 3$. In addition, we have  $\mathcal{G}(\bar{i},j)\cap \mathcal{G}(i,\bar{j})=\emptyset$.  By the maximality of $|\mathcal{G}|$, we have
$\mathcal{G}(i,j)=\binom{[n]\backslash \{i,j\}}{k-2}$.
Note that $\mathcal{G}(i,\bar{j})\subseteq\binom{[n]\backslash \{i,j\}}{k-1}$ and $\mathcal{G}(\bar{i},j)\subseteq\binom{[n]\backslash \{i,j\}}{k-1}$ are cross-intersecting.
Note that the inequality in Lemma \ref{W232} holds strictly
for $\mathcal{G}(i,\bar{j})$ and  $\mathcal{G}(\bar{i},j)$. Then 
\begin{align*}
|\mathcal{F}|&=|\mathcal{G}(i,j)|+|\mathcal{G}(i,\bar{j})|+|\mathcal{G}(\bar{i}, j)|\\ &\leq \binom{n-2}{k-2}+\binom{n-2}{k-1}-\binom{n-k-1}{k-1}-\binom{n-k-2}{k-2}-\binom{n-k-3}{k-3}+2\\
&<\binom{n-1}{k-1}-\binom{n-k-1}{k-1}-\binom{n-k-2}{k-2}-\binom{n-k-3}{k-3}+3.
\end{align*}

{\bf Case 3.}
 Suppose that $\mathcal{F}$ is not shifted, and there exists $\mathcal{G} \subseteq\binom{[n]}{k}$ obtained from $\mathcal{F}$ by repeated shifting operations such that $\mathcal{G}$  is neither EKR nor HM, and  $\mathcal{G} \nsubseteq \mathcal{J}_2(n, k)$, $|\mathcal{G}|=|\mathcal{F}|$, but
$s_{i, j}(\mathcal{G})$ is HM.

First of all, $s_{i, j}(\mathcal{G})$ is HM at $i$.
Therefore, we obtain $|\mathcal{G}(\bar{i},\bar{j})|\leq 1$. By Case 2, we may assume that $s_{i, j}(\mathcal{G})$ is not EKR.  If $|\mathcal{G}(\bar{i},\bar{j})|=0$, then $|\mathcal{G}(i,\bar{j})|\geq 3$ and $|\mathcal{G}(\bar{i},j)|\geq 3$. In addition, we have  $|\mathcal{G}(\bar{i},j)\cap \mathcal{G}(i,\bar{j})|=1$. By Lemma \ref{W232}, the same argument with Case 2 works.

Now assume that $|\mathcal{G}(\bar{i},\bar{j})|= 1$.
 Since $\mathcal{G}(i,j)$ and $\mathcal{G}(\bar{i},\bar{j})$ are cross-intersecting, we have
$
|\mathcal{G}(i,j)|\leq \binom{n-2}{k-2}-\binom{n-k-2}{k-2}.
$
Furthermore, $|\mathcal{G}(i,\bar{j})|\geq 3-|\mathcal{G}(\bar{i},\bar{j})|=2$ and $|\mathcal{G}(\bar{i},j)|\geq 2$.  Observe that $\mathcal{G}(i,\bar{j})$ and $\mathcal{G}(\bar{i},j)$ are cross-intersecting. 
Since $|\mathcal{G}(\bar{i},j)\cap \mathcal{G}(i,\bar{j})|=0$, the inequality in Lemma \ref{W231} holds strictly for 
 $\mathcal{G}(i,\bar{j})$ and  $\mathcal{G}(\bar{i},j)$.  We get 
$
|\mathcal{G}(i,\bar{j})|+|\mathcal{G}(\bar{i}, j)| \leq\binom{n-2}{k-1}-\binom{n-k-1}{k-1}-\binom{n-k-2}{k-2}+1.
$
Consequently,
\begin{align*}
|\mathcal{F}|&=|\mathcal{G}(i,j)|+|\mathcal{G}(i,\bar{j})|+|\mathcal{G}(\bar{i}, j)|+1\\
&\leq \binom{n-2}{k-2}-\binom{n-k-2}{k-2}+\binom{n-2}{k-1}-\binom{n-k-1}{k-1}-\binom{n-k-2}{k-2}+2\\
&<\binom{n-1}{k-1}-\binom{n-k-1}{k-1}-\binom{n-k-2}{k-2}-\binom{n-k-3}{k-3}+3.
\end{align*}

{\bf Case 4.}
 Suppose that $\mathcal{F}$ is not shifted, and there exists $\mathcal{G} \subseteq\binom{[n]}{k}$ obtained from $\mathcal{F}$ by repeated shifting operations such that $\mathcal{G}$  is neither EKR nor HM, and  $\mathcal{G} \nsubseteq \mathcal{J}_2(n, k)$, $|\mathcal{G}|=|\mathcal{F}|$, but
$s_{i, j}(\mathcal{G})\subseteq \mathcal{J}_2(n, k)$.

Since $s_{i, j}(\mathcal{G})\subseteq \mathcal{J}_2(n, k)$, we have $\mathcal{G}(\bar{i},\bar{j})\subseteq \{E_1, E_2\} \text{ for some } |E_1\cap E_2|=k-1$ and so $|\mathcal{G}(\bar{i},\bar{j})|\leq 2$.
By Cases 2 and 3, we may assume that $s_{i, j}(\mathcal{G})$ is neither EKR nor HM. Therefore, if  $|\mathcal{G}(\bar{i},\bar{j})|=0$, then $|\mathcal{G}(\bar{i},j)\cap \mathcal{G}(i,\bar{j})|=2$.
In addition,  $|\mathcal{G}(i,\bar{j})|\geq 3$ and $|\mathcal{G}(\bar{i},j)|\geq 3$. By Lemma \ref{W232}, the same argument with Case 2 works. If  $|\mathcal{G}(\bar{i},\bar{j})|=1$, then $|\mathcal{G}(\bar{i},j)\cap \mathcal{G}(i,\bar{j})|=1$. In addition, we have $|\mathcal{G}(i,\bar{j})|\geq 2$ and $|\mathcal{G}(\bar{i},j)|\geq 2$. By Lemma \ref{W231}, the same argument with Case 3 works.

Now assume that  $|\mathcal{G}(\bar{i},\bar{j})|=2$. Then
$\mathcal{G}(\bar{i},\bar{j})= \{E_1, E_2\} \text{ for some } |E_1\cap E_2|=k-1$.
Considering the following families:
\begin{align*}
\mathcal{G}_1 &=\left\{G \in\binom{[n]\backslash \{i,j\}}{k-2}: G \cap E_1=\emptyset \right\}, \\ 
\mathcal{G}_2&=\left\{G \in\binom{[n]\backslash \{i,j\}}{k-2}: G \cap E_1 \neq \emptyset, G \cap E_2=\emptyset\right\}.
\end{align*}
 Clearly, $|\mathcal{G}_1|=\binom{n-k-2}{k-2}$ and $|\mathcal{G}_2|=\binom{n-k-3}{k-3}$.
Note that $\mathcal{G}_1\cap \mathcal{G}_{2}=\emptyset$, $\mathcal{G}(i,j)$ and $\mathcal{G}(\bar{i},\bar{j})$ are cross-intersecting. Then we have
$$
|\mathcal{G}(i,j)|\leq \binom{n-2}{k-2}-|\mathcal{G}_1|-|\mathcal{G}_2|=\binom{n-2}{k-2}-\binom{n-k-2}{k-2}-\binom{n-k-3}{k-3}.
$$
Moreover, we have $|\mathcal{G}(\bar{i},j)\cap \mathcal{G}(i,\bar{j})|=0$, $|\mathcal{G}(i,\bar{j})|\geq 3-|\mathcal{G}(\bar{i},\bar{j})|=1$ and $|\mathcal{G}(\bar{i},j)|\geq 1$.
Observe that $\mathcal{G}(i,\bar{j})$ and $\mathcal{G}(\bar{i},j)$ are cross-intersecting. The inequality in Theorem \ref{mainH} holds strictly. 
We get
$
|\mathcal{G}(i,\bar{j})|+|\mathcal{G}(\bar{i}, j)| \leq\binom{n-2}{k-1}-\binom{n-k-1}{k-1}.
$
Therefore, we obtain
\begin{align*}
|\mathcal{F}|&=|\mathcal{G}(i,j)|+|\mathcal{G}(i,\bar{j})|+|\mathcal{G}(\bar{i}, j)|+2\\
&\leq \binom{n-2}{k-2}-\binom{n-k-2}{k-2}-\binom{n-k-3}{k-3}+\binom{n-2}{k-1}-\binom{n-k-1}{k-1}+2\\
&<\binom{n-1}{k-1}-\binom{n-k-1}{k-1}-\binom{n-k-2}{k-2}-\binom{n-k-3}{k-3}+3,
\end{align*}
as desired. This completes the proof.$\hfill \square$\vspace{3mm}

It remains to prove Theorem \ref{H24} for $k=4$. We shall establish the corresponding versions of Theorem \ref{51} and Lemma \ref{W232} for  $k=3$. 
We defer the detailed discussions to Appendix \ref{App-B}.

The \textit{covering number} of a family is the  size of the smallest set that intersects all sets from the family.
There are also some literature \cite{F80, FT251,FT252, K24} studying the intersecting families with covering number at least three or five. Therefore, it is natural to investigate the following problem.

\begin{problem}
  What kind of inequalities on cross-intersecting families would be needed to provide new proofs of the results on intersecting families with covering number at least three?  
\end{problem}

\section*{Acknowledgement} 
We would  like to thank Russ Woodroofe for sharing the reference \cite{BW2024}, and Andrey Kupavskii for sharing \cite{Kup2024}.  We would also like to thank the referees for the valuable suggestions which greatly improved the presentation of the manuscript.
 Yongtao Li was supported by the Postdoctoral Fellowship Program of CPSF (No. GZC20233196). Lihua Feng was supported by the NSFC (Nos. 12271527 and 12471022). Guihai Yu was supported by the NSFC (Nos. 12461062 and  11861019) and Natural Science Foundation of Guizhou (Nos. [2020]1Z001 and  [2021]5609).

\appendix
\section{The case $k=3$ in the Han--Kohayakawa theorem}

\label{App-A}

Define  $\mathcal{A}=\{\{1,2\} \cup\{i\}: i \in[3, n]\}\cup\left\{G \in\binom{[n]}{2}:  G\cap \{1,2\}\neq \emptyset\right\}.$

\begin{lemma}\label{HK31}
 Let $n \geq 6$ be an integer.  Let $\mathcal{F} \subseteq\binom{[n]}{3}$ and $\mathcal{G} \subseteq\binom{[n]}{2}$ be cross-intersecting families. Suppose that $|\mathcal{F}| \geq 2$ and $\mathcal{F}$ is $2$-intersecting.  If $\mathcal{F}\cup \mathcal{G}$ is not isomorphic to a subfamily of $\mathcal{A}$, then
$$
|\mathcal{F}|+|\mathcal{G}| \leq\binom{n}{2}-\binom{n-3}{2}-\binom{n-4}{1}+2.
$$
The equality holds if and only if $\mathcal{F}=\left\{F_1, F_2\right\}$ for some $F_1, F_2 \in\binom{[n]}{3}$ with $\left|F_1 \cap F_2\right|=2$ and $\mathcal{G}=\left\{G \in\binom{[n]}{2}: G \cap F_1 \neq \emptyset \text{ and } G \cap F_2 \neq \emptyset\right\}$.
\end{lemma}
\begin{proof}
If $|\mathcal{F}|=2$, then the result holds clearly.
Now suppose that  $|\mathcal{F}|\geq 3$. Let $x_1, \ldots, x_5\in [n]$ be five different numbers. If there exists $\left\{\{x_1, x_2, x_3\}, \{x_1, x_2, x_4\}, \{x_1, x_2, x_5\}\right\}\subseteq \mathcal{F}$, then
$
\mathcal{G}\subseteq \{\{x_1,i\}, i\in [n]\backslash \{x_1\} \}\cup \{\{x_2,i\}, i\in [n]\backslash \{x_1,x_2\} \}.$
Since $\mathcal{F}\cup \mathcal{G}$ is not isomorphic to a subfamily of $\mathcal{A}$, we have $\{a, b,c\}\in \mathcal{F}$ for three different elements $a, b, c\in [n]$ satisfying $|\{x_1, x_2\}\cap \{a, b,c\}|\leq 1$.  Since $|\{a, b,c\}\cap\{x_1, x_2, x_3\}|\geq 2$ and $|\{a, b,c\}\cap\{x_1, x_2, x_4\}|\geq 2$, we get $\{a, b,c\}=\{x_1, x_3, x_4\}$ or $\{a, b,c\}=\{x_2, x_3, x_4\}$. But  $|\{x_1, x_3, x_4\}\cap\{x_1, x_2, x_5\}|=1$ and
$|\{x_2, x_3, x_4\}\cap\{x_1, x_2, x_5\}|=1$, a contradiction.
Therefore, we may assume that $\left\{\{x_1, x_2, x_3\}, \{x_1, x_2, x_4\}, \{x_1, x_3, x_4\}\right\}\subseteq \mathcal{F}$. Then
$
\mathcal{G}\subseteq \{\{x_1,i\}, i\in [n]\backslash \{x_1\} \}\cup \{\{x_2,x_3\}, \{x_2,x_4\}, \{x_3,x_4\}\}.
$
So $|\mathcal{G}|\leq n+2$. Since $\mathcal{F}$ is $2$-intersecting, we must have $\mathcal{F}\subseteq\left\{\{x_1, x_2, x_3\}, \{x_1, x_2, x_4\}, \{x_1, x_3, x_4\}, \{x_2, x_3, x_4\}\right\}$. Thus $|\mathcal{F}| \leq 4\leq n-2$.
It follows that
$|\mathcal{F}|+|\mathcal{G}| \leq 2n=\binom{n}{2}-\binom{n-3}{2}-\binom{n-4}{1}+2$. It is not hard to see that the inequality holds strictly. This completes the proof.
\end{proof}

Define  $\mathcal{B}=\{\{1\} \cup\{i\}: i \in[2, n]\}$.
The next lemma deals with the case $k=2$ for Lemma \ref{W231}.

\begin{lemma}\label{HK32}
Let $n \geq 4$ be an integer. Let $\mathcal{F} \subseteq\binom{[n]}{2}$ and $\mathcal{G} \subseteq\binom{[n]}{2}$ be  cross-intersecting families. Suppose that $|\mathcal{G}| \geq|\mathcal{F}| \geq 2$ and $\mathcal{F}\cup \mathcal{G}$ is not isomorphic to a subfamily of $\mathcal{B}$. Then
$$
|\mathcal{F}|+|\mathcal{G}| \leq\binom{n}{2}-\binom{n-2}{2}-\binom{n-3}{1}+2.
$$
For $n \geq 5$, the equality holds if and only if $\mathcal{F}=\left\{F_1, F_2\right\}$ for some $F_1, F_2 \in\binom{[n]}{2}$ with $\left|F_1 \cap F_2\right|=1$ and $\mathcal{G}=\left\{G \in\binom{[n]}{2}: G \cap F_1 \neq \emptyset \text{ and } G \cap F_2 \neq \emptyset\right\}$.
\end{lemma}
\begin{proof}
First note that
$
\binom{n}{2}-\binom{n-2}{2}-\binom{n-3}{1}+2=n+2.
$
If $|\mathcal{F}|=2$, then the result holds clearly.
So we may assume that  $|\mathcal{F}|\geq 3$. Let $x_1, \ldots, x_6\in [n]$ be six different numbers (if exist). If there exists $\left\{\{x_1, x_2\}, \{x_3, x_4\}, \{x_5, x_6\}\right\}\subseteq \mathcal{F}$, then $\mathcal{G}=\emptyset$, a contradiction. So we may assume that $\left\{\{x_1, x_2\}, \{x_1, x_3\}\right\}\subseteq \mathcal{F}$, note that $\mathcal{F}\cup \mathcal{G}$ is not isomorphic to a subfamily of $\mathcal{B}$, so
there are some $u\neq v \in  [n]\backslash \{x_1\}$ such that $\{u, v\}\in  \mathcal{F}$ or  $\{u, v\}\in  \mathcal{G}$.

If $\{u, v\}\in  \mathcal{F}$, then $ \mathcal{G}\subseteq\left\{\{x_1, u\}, \{x_1, v\}, \{x_2, x_3\}\right\}$. So $|\mathcal{G}|\leq 3$. Then $|\mathcal{F}|+|\mathcal{G}| \leq 6\leq n+2$.

If $\{u, v\}\in  \mathcal{G}$, then $\{u, v\}=\{x_2, x_3\}$. Observe that
$
\mathcal{G}\subseteq \{\{x_1,i\}, i\in [n]\backslash \{x_1\} \}\cup  \{\{x_2, x_3\}\}
$.
Since $|\mathcal{G}| \geq|\mathcal{F}| \geq 3$, we have $\left\{\{x_1, a\}, \{x_1, b\}, \{x_2, x_3\}\right\}\subseteq \mathcal{G}$ for some $a\neq b \in  [n]\backslash \{x_1\}$. Then the cross-intersecting property of $\mathcal{F}$ and $\mathcal{G}$ implies that $ \mathcal{F}\subseteq\left\{\{x_1, x_2\}, \{x_1, x_3\}, \{a, b\}\right\}$. Since $|\mathcal{F}|\geq 3$, we have $ \mathcal{F}=\left\{\{x_1, x_2\}, \{x_1, x_3\}, \{a, b\}\right\}$ and $a\in \{x_2, x_3\}$ or $b\in \{x_2, x_3\}$. It follows that $\mathcal{G}=\left\{\{x_1, a\}, \{x_1, b\}, \{x_2, x_3\}\right\}$. Consequently, $|\mathcal{F}|+|\mathcal{G}| \leq 6\leq n+2$. This completes the proof.
\end{proof}

\begin{prop}\label{pr3}
Let $n \geq 5$ be an integer. Let $\mathcal{F} \subseteq\binom{[n]}{2}$ and $\mathcal{G} \subseteq\binom{[n]}{2}$ be  cross-intersecting families and $|\mathcal{F} \cap \mathcal{G}|\leq 1$. Suppose that $|\mathcal{F}| \geq 2$ and $|\mathcal{G}| \geq 2$. Then
$$
|\mathcal{F}|+|\mathcal{G}| \leq\binom{n}{2}-\binom{n-2}{2}-\binom{n-3}{1}+1.
$$
\end{prop}
\begin{proof}
If $\mathcal{F}\cup \mathcal{G}$ is not isomorphic to a subfamily of $\{\{1\} \cup\{i\}: i \in[2, n]\}$, then the result holds by Lemma \ref{HK32}.
If $\mathcal{F}\cup \mathcal{G}$ is isomorphic to a subfamily of $\{\{1\} \cup\{i\}: i \in[2, n]\}$, note that $|\mathcal{F}|+|\mathcal{G}|=|\mathcal{F}\cup \mathcal{G}|+|\mathcal{F}\cap \mathcal{G}|$, then $|\mathcal{F} \cap \mathcal{G}|\leq 1$ implies that
$
|\mathcal{F}|+|\mathcal{G}| \leq n<\binom{n}{2}-\binom{n-2}{2}-\binom{n-3}{1}+1,$
as required.
\end{proof}

\noindent{\bf Proof of Theorem \ref{H17} for $k=3$.}
Suppose that $\mathcal{F}\subseteq\binom{[n]}{3}$ is neither EKR nor HM, and not a subfamily of $\mathcal{G}_2(n, 3)$. In addition, $|\mathcal{F}|$ is maximal.
Since $\mathcal{F}$ is not a subfamily of $\mathcal{G}_2(n, 3)$, we have that $\mathcal{F}(\bar{1})\cup\mathcal{F}(1)$ is not isomorphic to a subfamily of $\mathcal{A}$.  According to Lemma \ref{HK31}, Proposition \ref{pr3} and the proof of Theorem \ref{H17} for $k\geq 4$, we only need to deal with the following case.

{\bf Case 1.}
Suppose that $\mathcal{F}$ is not shifted, and there exists $\mathcal{G} \subseteq\binom{[n]}{3}$ obtained from $\mathcal{F}$ by repeated shifting operations such that $\mathcal{G}$  is neither EKR nor HM, and $\mathcal{G}\nsubseteq \mathcal{G}_2(n, 3)$, $|\mathcal{G}|=|\mathcal{F}|$, but
$s_{i, j}(\mathcal{G})\subseteq \mathcal{G}_2(n, 3)$.

By conditions, there exist three different elements $x_0, x_1, x_2$ such that
$$
s_{i, j}(\mathcal{G})\subseteq \left\{G \in\binom{[n]}{3}: \{x_1,x_2\} \subseteq G\right\} \cup\left\{G \in\binom{[n]}{3}: x_0 \in G, G \cap \{x_1,x_2\} \neq \emptyset\right\}.
$$
Hence, $i\in \{x_0, x_1, x_2\}$ and $j\notin \{x_0, x_1, x_2\}$. By symmetry, let us assume that $x_0=i$. This implies that $\mathcal{G}(i,j)\subseteq \{\{x_1\}, \{x_2\}\}$ and so $|\mathcal{G}(i,j)|\leq 2$.
Moreover, $\mathcal{G}(\bar{i},\bar{j})\subseteq \{\{x_1, x_2, y\}: y\in [n]\backslash\{x_1, x_2, i, j\}\}$ and so $|\mathcal{G}(\bar{i},\bar{j})|\leq n-4$.
Observe that $\mathcal{G}(i,\bar{j})\cap \mathcal{G}(\bar{i},j)=\emptyset \text{ or } \{x_1, x_2\}$.
By the maximality of $|\mathcal{G}|$, we may assume that $\mathcal{G}(i,\bar{j})\cap \mathcal{G}(\bar{i},j)=\{x_1, x_2\}$.
Since $\mathcal{G}$ is not a subfamily of $\mathcal{G}_2(n, 3)$, we have
$\mathcal{G}(i,\bar{j})\neq \{x_1, x_2\}$ and $\mathcal{G}(\bar{i},j)\neq  \{x_1, x_2\}$.
Then  $|\mathcal{G}(i,\bar{j})|\geq 2$ and $|\mathcal{G}(\bar{i},j)|\geq 2$.  Note that $\mathcal{G}(i,\bar{j})$ and $\mathcal{G}(\bar{i},j)$ are cross-intersecting. By Proposition \ref{pr3}, we obtain
\begin{align*}
|\mathcal{F}| &=|\mathcal{G}(i,j)|+|\mathcal{G}(\bar{i},\bar{j})|+|\mathcal{G}(i,\bar{j})|+|\mathcal{G}(\bar{i}, j)| \\ 
& \leq 2+n-4+\binom{n-2}{2}-\binom{n-4}{2}-\binom{n-5}{1}+1\\
&<\binom{n-1}{2}-\binom{n-4}{2}-\binom{n-5}{1}+2,
\end{align*}
as required.$\hfill \square$\vspace{3mm}

\section{The case $k=4$ in the Huang--Peng theorem}

\label{App-B}

Denote $
\mathcal{C}=\{[3] \cup\{i\}: i \in[4, n]\}\cup\left\{G \in\binom{[n]}{3}:  G\cap [3]\neq \emptyset\right\}.$


\begin{lemma}\label{HP51}
 Let $n \geq 7$ be an integer.  Let $\mathcal{F} \subseteq\binom{[n]}{4}$ and $\mathcal{G} \subseteq\binom{[n]}{3}$ be cross-intersecting families. Suppose that $|\mathcal{F}| \geq 3$ and $\mathcal{F}$ is $2$-intersecting.
In addition,  $\mathcal{F}$ and $\mathcal{G}$ are both shifted. If $\mathcal{F}\cup \mathcal{G}$ is not isomorphic to a subfamily of $\mathcal{C}$, then
$$
|\mathcal{F}|+|\mathcal{G}| \leq\binom{n}{3}-\binom{n-4}{3}-\binom{n-5}{2}-\binom{n-6}{1}+3.
$$
For $n \geq 8$, the equality holds if and only if $\mathcal{F}=\left\{F_1, F_2, F_3\right\}$ with $F_i:=[3]\cup \{3+i\}$ for each $i\in [3]$  and $\mathcal{G}=\left\{G \in\binom{[n]}{3}: G \cap F_i \neq\emptyset \text{ for any } i=1, 2, 3\right\}$, under isomorphism.
\end{lemma}
\begin{proof}
We proceed by applying induction on  $n\geq 7$. For the case $n=7$, the upper bound in Lemma \ref{HP51} is  $\binom{7}{3}$. Since $\mathcal{F}$ and $\mathcal{G}$ are cross-intersecting, at most half of the sets of $\binom{[7]}{4}\cup \binom{[7]}{3}$ belong to $\mathcal{F}\cup\mathcal{G}$.
Thus
$|\mathcal{F}|+|\mathcal{G}| \leq\binom{7}{3}$. 
Suppose that $n \geq 8$ and the result holds for integers less than $n$. Since $\mathcal{F}$ and $\mathcal{G}$ are shifted, by Lemma \ref{S2}, we have $|\mathcal{F}(\bar{n})| \geq 3$. Clearly, $\mathcal{F}(\bar{n})$ is $2$-intersecting. Note that $\mathcal{F}(\bar{n})$ and
$\mathcal{G}(\bar{n})$ are cross-intersecting. By the  induction hypothesis, we get
$$
|\mathcal{F}(\bar{n})|+|\mathcal{G}(\bar{n})| \leq\binom{n-1}{3}-\binom{n-5}{3}-\binom{n-6}{2}-\binom{n-7}{1}+3.
$$

If $|\mathcal{F}(n)|=0$, then the same argument as Case 1 of Theorem \ref{51} works.

If $|\mathcal{F}(n)|=1$, let $\mathcal{F}^{\prime}=\{[3]\cup\{i\}: i\in [4, n-1]\}$. Then $\mathcal{F}^{\prime}\subseteq\mathcal{F}(\bar{n})$. If $\mathcal{F}^{\prime}=\mathcal{F}(\bar{n})$, then $\mathcal{F}=\{[3]\cup\{i\}: i\in [4, n]\}$. So $\mathcal{G} \subseteq\left\{G \in\binom{[n]}{3}: G\cap[3] \neq \emptyset\right\}$. But $\mathcal{F}\cup \mathcal{G}$ is isomorphic to a subfamily of $\mathcal{C}$, a contradiction.
Hence, $\mathcal{F}^{\prime}\subsetneq \mathcal{F}(\bar{n})$. Then we have $\mathcal{G}(n) \subsetneq\left\{G^*\in\binom{[n-1]}{2}: G^*\cap[3] \neq \emptyset\right\}$. This implies that
$
|\mathcal{G}(n)| <\binom{n-1}{2}-\binom{n-4}{2}.
$
Consequently,
$$
|\mathcal{F}|+|\mathcal{G}|=|\mathcal{F}(\bar{n})|+|\mathcal{G}(\bar{n})|+|\mathcal{G}(n)|+1 <\binom{n}{3}-\binom{n-4}{3}-\binom{n-5}{2}-\binom{n-6}{1}+4.
$$
That is, $
|\mathcal{F}|+|\mathcal{G}| \leq\binom{n}{3}-\binom{n-4}{3}-\binom{n-5}{2}-\binom{n-6}{1}+3.$
Note that $\mathcal{F}(\bar{n})$ and $\mathcal{G}(n)$ are shifted on $[n-1]$. If the equality holds, then $\mathcal{G}(n)=\left\{G \in\binom{[n-1]}{2}: G\cap[3] \neq \emptyset\right\}\backslash \{3,n-1\}.$
Since $\mathcal{F}^{\prime}\subsetneq\mathcal{F}(\bar{n})$ and $\mathcal{F}(\bar{n})$  is  shifted, we have $\{1,2,4, 5\}\in \mathcal{F}(\bar{n})$. Observe that $\{3,n-2\}\in \mathcal{G}(n)$ and $n-2\geq 6$. So $\{1,2,4, 5\}\cap\{3,n-2\}=\emptyset$. But $\mathcal{F}(\bar{n})$ and $\mathcal{G}(n)$ are cross-intersecting, a contradiction. Therefore, we have $
|\mathcal{F}|+|\mathcal{G}| <\binom{n}{3}-\binom{n-4}{3}-\binom{n-5}{2}-\binom{n-6}{1}+3.
$

If $|\mathcal{F}(n)|\geq 2$, since $\mathcal{F}\cup \mathcal{G}$ is not isomorphic to a subfamily of $\mathcal{C}$, we have that  $\mathcal{F}(n)\cup \mathcal{G}(n)$ is not isomorphic to a subfamily of $\mathcal{A}$. In addition, by Lemmas \ref{S1} and \ref{S4}, $\mathcal{F}(n)$ is $2$-intersecting, $\mathcal{F}(n)$ and $\mathcal{G}(n)$ are cross-intersecting.  Then applying Lemma \ref{HK31} to $\mathcal{F}(n)$ and $\mathcal{G}(n)$ yields
$
|\mathcal{F}(n)|+|\mathcal{G}(n)| \leq\binom{n-1}{2}-\binom{n-4}{2}-\binom{n-5}{1}+2.
$
Combining this and the previous inequality about $\mathcal{F}(\bar{n})$ and $\mathcal{G}(\bar{n})$, we obtain
\begin{align*}
|\mathcal{F}|+|\mathcal{G}|=&|\mathcal{F}(\bar{n})|+|\mathcal{G}(\bar{n})|+|\mathcal{F}(n)|+|\mathcal{G}(n)|\\
\leq&\binom{n}{3}-\binom{n-4}{3}-\binom{n-5}{2}-\binom{n-6}{1}+3-\binom{n-7}{1}\\
<&\binom{n}{3}-\binom{n-4}{3}-\binom{n-5}{2}-\binom{n-6}{1}+3,
\end{align*}
as required.
\end{proof}

Define $
\mathcal{P}=\{\{1,2\} \cup\{i\}: i \in[3, n]\}\cup\left\{G \in\binom{[n]}{3}:  G\cap \{1,2\}\neq \emptyset\right\}$ and $
\mathcal{Q}=\left\{Q \in\binom{[n]}{3}: 1\in Q\right\}$.

\begin{prop}\label{pr5}
Let $n \geq 7$ be an integer. Let $\mathcal{F} \subseteq\binom{[n]}{3}$ and $\mathcal{G} \subseteq\binom{[n]}{3}$ be  non-empty cross-intersecting families and $|\mathcal{F} \cap \mathcal{G}|\leq 2$. Suppose that $|\mathcal{F}| \geq 3$ and $|\mathcal{G}| \geq 3$. Then
$$
|\mathcal{F}|+|\mathcal{G}| \leq\binom{n}{3}-\binom{n-3}{3}-\binom{n-4}{2}-\binom{n-5}{1}+2.
$$
\end{prop}
\begin{proof}
If $\mathcal{F}\cup \mathcal{G}$ is isomorphic to a subfamily of $\mathcal{P}$ or $\mathcal{Q}$, then $|\mathcal{F}\cup \mathcal{G}|\leq \binom{n}{3}-\binom{n-2}{3}$ or $|\mathcal{F}\cup \mathcal{G}|\leq \binom{n-1}{2}$. Since $|\mathcal{F} \cap \mathcal{G}|\leq 2$, we obtain
\begin{align*}
|\mathcal{F}|+|\mathcal{G}|\leq \binom{n}{3}-\binom{n-2}{3}+2<\binom{n}{3}-\binom{n-3}{3}-\binom{n-4}{2}-\binom{n-5}{1}+2.
\end{align*}

Now suppose that $\mathcal{F}\cup \mathcal{G}$ is neither isomorphic to a subfamily of $\mathcal{P}$ nor a subfamily of $\mathcal{Q}$. If $|\mathcal{F}| = 3$ or $|\mathcal{G}| = 3$, the result holds clearly. So let us assume that $|\mathcal{F}| \geq 4$ and $|\mathcal{G}| \geq 4$.

Let $x_1, \ldots, x_5\in [n]$ be five different numbers.
When either $\mathcal{F}$ or $\mathcal{G}$ is a star. By symmetry, we may assume that $\mathcal{F}$ is a star at $x_1$.
If any four sets of $\mathcal{F}$ contain $x_1, x_2$, then $G\cap \{x_1,x_2\}\neq \emptyset$ for any $G\in \mathcal{G}$. So $\mathcal{F}\cup \mathcal{G}$ is isomorphic to a subfamily of $\mathcal{P}$, a contradiction.
If $\{\{x_1,x_2,x_3\}, \{x_1,x_2,x_4\}, \{x_1,x_2,x_5\}, \{x_1,a,b\}  \}\subseteq \mathcal{F}$ for some $a\neq b$ and $a,b\notin \{x_1, x_2\}$, then $
\mathcal{G}\subseteq \left\{G \in\binom{[n]}{3}: x_1\in G\right\}\cup\left\{\{x_2,a, i\}:  i\in [n]\backslash\{x_1,x_2,a\}\right\}
\cup\left\{\{x_2,b, i\}:  i\in [n]\backslash\{x_1,x_2,a,b\}\right\}$ $\cup $ 
$\{\{x_3,x_4,x_5\}\}.
$
It follows that
$
|\mathcal{F}|+|\mathcal{G}|=|\mathcal{F}\cup \mathcal{G}|+|\mathcal{F}\cap \mathcal{G}|\leq \binom{n-1}{2}+2n-4$,
which is smaller than the required upper bound.
If  $\{\{x_1,x_2,x_3\}, \{x_1,x_2,x_4\}, \{x_1,a,b\}, \{x_1,a,d\}  \}\subseteq \mathcal{F}$ for three different elements $a, b, d$ and $a\notin \{x_1, x_2, x_3, x_4\}$ and $b, d\notin\{x_1, x_2\}$, then
$
\mathcal{G}\subseteq\left\{G \in\binom{[n]}{3}: x_1\in G\right\}\cup\left\{\{x_2,a, i\}:  i\in [n]\backslash\{x_1,x_2,a\}\right\}\cup\{\{x_2,b,d\},\{x_3,x_4,a\}\}.
$
So $|\mathcal{F}|+|\mathcal{G}|=|\mathcal{F}\cup \mathcal{G}|+|\mathcal{F}\cap \mathcal{G}|\leq \binom{n-1}{2}+n+1$. This is also smaller than the required upper bound.
If $\{\{x_1,x_2,x_3\}, \{x_1,x_2,x_4\}$, $\{x_1,a,b\}$, $\{x_1,c,d\}  \}\subseteq \mathcal{F}$ for four different elements $a, b, c, d$ and $a, b,c, d\notin \{x_1, x_2\}$, then we similarly have $|\mathcal{F}|+|\mathcal{G}|\leq \binom{n-1}{2}+6$. This is also smaller than the required upper bound.
If $\{\{x_1,x_2,x_3\}$, $\{x_1,x_4,x_5\}$, $\{x_1,x_6,x_7\}$, 
$\{x_1,x_8,x_9\}  \}\subseteq \mathcal{F}$ for nine different elements $x_1, \ldots, x_9 \in [n]$, then $G\cap \{x_1\}\neq \emptyset$ for any $G\in \mathcal{G}$. So $\mathcal{F}\cup \mathcal{G}$ is isomorphic to a subfamily of $\mathcal{Q}$, a contradiction.

It remains to consider the case that  $\mathcal{F}$ and $\mathcal{G}$ are both non-trivial. Although we can get the conclusion by a similar discussion as above. For simplicity, we use an old result due to Frankl and Tokushige.

\begin{lemma}[See \cite{F98, F24}] \label{F24}
Let $n \geq 2k+1\geq 5$ be an integer. Let $\mathcal{F} \subseteq\binom{[n]}{k}$ and $\mathcal{G} \subseteq\binom{[n]}{k}$ be non-trivial cross-intersecting families. Then
$$
|\mathcal{F}|+|\mathcal{G}| \leq\binom{n}{k}-2\binom{n-k}{k}+\binom{n-2k}{k}+2.
$$
For $k\geq 3$, the equality holds if and only if $\mathcal{F}=\left\{F_1, F_2\right\}$ for some disjoint $F_1, F_2 \in\binom{[n]}{k}$ and $\mathcal{G}=\left\{G \in\binom{[n]}{k}: G \cap F_i \neq\emptyset \text{ for any } i=1, 2\right\}$.
\end{lemma}
Note that for $ 2k+1\leq n\leq 3k-1$, the upper bound of Lemma \ref{F24} is
reduced to $\binom{n}{k}-2\binom{n-k}{k}+2$. Returning to our proof. Note that $|\mathcal{F}|\geq 4$. By Lemma \ref{F24}, we have
$
|\mathcal{F}|+|\mathcal{G}| \leq\binom{n}{3}-2\binom{n-3}{3}+1
$
for $ 7 \leq n\leq 8$,
and
$
|\mathcal{F}|+|\mathcal{G}| \leq\binom{n}{3}-2\binom{n-3}{3}+\binom{n-6}{3}+1
$
for $ n\geq 9$. By simple calculations, we know that these bounds are smaller or equal to the required upper bound.
This completes the proof of Proposition \ref{pr5}.
\end{proof}

\begin{lemma}\label{Lem52}
 Let $n \geq 7$ be an integer.  Let $x_1\neq x_2\in [n]$ and $\mathcal{F}, \mathcal{G} \subseteq\left\{A \in\binom{[n]}{3}:  A\cap \{x_1, x_2\}\neq \emptyset\right\}$ be cross-intersecting families. Let $
\mathcal{R}=\left\{A \in\binom{[n]}{3}: \{x_1,x_2\} \subseteq A \right\}$.
Suppose that $\mathcal{R}\subsetneq \mathcal{F}$, $\mathcal{R}\subsetneq \mathcal{G}$ and $(\mathcal{F}\backslash\mathcal{R}) \cap(\mathcal{G}\backslash\mathcal{R}) =\emptyset$.
Then
$$
|\mathcal{F}|+|\mathcal{G}| \leq\binom{n}{3}-\binom{n-3}{3}-\binom{n-4}{2}-\binom{n-5}{1}+2.
$$
\end{lemma}
\begin{proof}
By conditions and symmetry, we may assume that $\{x_1, a, b\}\in \mathcal{F}$ for some
$a\neq b\in [n]\backslash \{x_1, x_2\}$. Then
\begin{align*}
\mathcal{G}\subseteq\left\{G \in\binom{[n]}{3}: x_1\in G\right\}\cup\left\{\{x_2,a, i\}:  i\in [n]\backslash\{x_1,x_2,a\}\right\}\cup\left\{\{x_2,b, i\}:  i\in [n]\backslash\{x_1,x_2,a, b\}\right\}.
\end{align*}
If $\mathcal{F}$ is a star at $x_1$, or $\mathcal{F}$ is a star at $x_1$ and there is at most a $\{x_2, c, d\}\in \mathcal{F}$ for some
$c\neq d\in [n]\backslash \{x_1, x_2\}$, note that $(\mathcal{F}\backslash\mathcal{R}) \cap(\mathcal{G}\backslash\mathcal{R}) =\emptyset$, then we have
$
|\mathcal{F}|+|\mathcal{G}|\leq n-2+1+ \binom{n-1}{2}+n-3+n-4
=\binom{n-1}{2}+3n-8
$.
Note that $\binom{n}{3}-\binom{n-3}{3}-\binom{n-4}{2}-\binom{n-5}{1}+2=\binom{n-1}{2}+\binom{n-2}{2}+3$.
For $n\geq 7$, let
$
g(n)=\binom{n-1}{2}+\binom{n-2}{2}+3-\left(\binom{n-1}{2}+3n-8\right)
=\binom{n-2}{2}-3n+11=\frac{1}{2}\left(n^2-11n+28 \right)
$.
Then $h(n)\geq 0$. So the result follows. By symmetry, the same case holds for $\mathcal{G}$.

Next we consider the following three cases.

(i) If there exist $\{x_2, c, d\}, \{x_2, c, e\}\in \mathcal{F}$ for three different elements $c, d, e\in [n]\backslash \{x_1, x_2\}$, and $\{x_1, a, b\}, \{x_1, a, y\}\in \mathcal{F}$ for three different elements $a, b, y\in [n]\backslash \{x_1, x_2\}$, then
\begin{align*}
\mathcal{G}\subseteq&\mathcal{R}\cup\left\{\{x_1,c, i\}:  i\in [n]\backslash\{x_1,x_2,c\}\right\}\cup\{\{x_1,e, d\}, \{x_2,b, y\}\}\cup\left\{\{x_2,a, i\}:  i\in [n]\backslash\{x_1,x_2,a\}\right\}.
\end{align*}
So $|\mathcal{G}|\leq 3n-6$.

(ii) If there exist $\{x_2, c, d\}, \{x_2, e, f\}\in \mathcal{F}$ for four different elements $c, d, e, f\in [n]\backslash \{x_1, x_2\}$, and $\{x_1, a, b\}, \{x_1, a, y\}\in \mathcal{F}$ for three different elements $a, b, y\in [n]\backslash \{x_1, x_2\}$, then
\begin{align*}
\mathcal{G}\subseteq&\mathcal{R}\cup\left\{\{x_1,c, i\}:  i\in \{e,f\}\right\}\cup\left\{\{x_1,d, i\}:  i\in \{e,f\}\right\}\cup\left\{\{x_2,a, i\}:  i\in [n]\backslash\{x_1,x_2,a\}\right\}\\
&\cup\{\{x_2,b, y\}\}.
\end{align*}
So $|\mathcal{G}|\leq 2 n$.

(iii) If there exist $\{x_2, c, d\}, \{x_2, e, f\}\in \mathcal{F}$ for four different elements $c, d, e, f\in [n]\backslash \{x_1, x_2\}$, and $\{x_1, a, b\}, \{x_1, y, z\}\in \mathcal{F}$ for four different elements $a, b, y, z\in [n]\backslash \{x_1, x_2\}$, then
\begin{align*}
\mathcal{G}\subseteq&\mathcal{R}\cup\left\{\{x_1,c, i\}:  i\in \{e,f\}\right\}\cup\left\{\{x_1,d, i\}:  i\in \{e,f\}\right\}\cup\left\{\{x_2,a, i\}:  i\in \{y,z\}\right\}\\
&\cup\left\{\{x_2,b, i\}:  i\in \{y,z\}\right\}.
\end{align*}
So $|\mathcal{G}|\leq n+6$.

Returning to case (i), we shall assume that cases (ii) and (iii) are not true. Then  $|\mathcal{F}|\leq n-2+2(n-3)=3n-8$. Therefore, we have
$|\mathcal{F}|+|\mathcal{G}|\leq 6n-14\leq\binom{n-1}{2}+\binom{n-2}{2}+3$.
For case (ii), we shall assume that case (iii) is not true. In addition, we may assume that at least two elements of $\left\{\{x_1,c, i\}:  i\in \{e,f\}\right\}\cup\left\{\{x_1,d, i\}:  i\in \{e,f\}\right\}$ are in $\mathcal{G}$.
Then  $|\mathcal{F}|\leq n-2+n-2+n-3=3n-7$. Therefore, we have
$|\mathcal{F}|+|\mathcal{G}|\leq 5n-7\leq\binom{n-1}{2}+\binom{n-2}{2}+3$.
For case (iii), we may assume that at least two elements of $\left\{\{x_1,c, i\}:  i\in \{e,f\}\right\}\cup\left\{\{x_1,d, i\}:  i\in \{e,f\}\right\}$ are in $\mathcal{G}$, and at least two elements of $\left\{\{x_2,a, i\}:  i\in \{y,z\}\right\}\cup\left\{\{x_2,b, i\}:  i\in \{y,z\}\right\}$ are in $\mathcal{G}$.  Then  $|\mathcal{F}|\leq n-2+n-2+n-2=3n-6$. Therefore, we have
$|\mathcal{F}|+|\mathcal{G}|\leq 4n\leq\binom{n-1}{2}+\binom{n-2}{2}+3$.
This completes the proof.
\end{proof}

\noindent{\bf Proof of Theorem \ref{H24} for $k=4$.}
Let $\mathcal{F} \subseteq\binom{[n]}{4}$ be an maximum intersecting family which is neither EKR nor HM, and $\mathcal{F} \nsubseteq \mathcal{J}_2(n, 4),\mathcal{G}_2(n, 4), \mathcal{G}_3(n, 4)$. For any $x \in[n]$, let $|\mathcal{F}(\bar{x})|\geq 3$. Furthermore, 
for any $\mathcal{G}$ obtained from $\mathcal{F}$ by repeated shifting operations such that $\mathcal{G}$  is neither EKR nor HM, and  $\mathcal{G} \nsubseteq \mathcal{J}_2(n, 4), \mathcal{G}_2(n, 4), \mathcal{G}_3(n, 4)$, we may assume that $|\mathcal{G}(\bar{x})|\geq 3$.
Since $\mathcal{F}$ is not a subfamily of $\mathcal{G}_3(n, 4)$, we have that $\mathcal{F}(\bar{1})\cup \mathcal{F}(1)$ is  not isomorphic to a subfamily of $\mathcal{C}$.  According to Lemma \ref{HP51}, Proposition \ref{pr5} and the proof of Theorem \ref{H24} for $k\geq 5$, we only need to deal with the following two cases.

{\bf Case 1.}
 Suppose that $\mathcal{F}$ is not shifted, and there exists $\mathcal{G} \subseteq\binom{[n]}{4}$ obtained from $\mathcal{F}$ by repeated shifting operations such that $\mathcal{G}$  is neither EKR nor HM, and  $\mathcal{G} \nsubseteq \mathcal{J}_2(n, 4)$, $\mathcal{G}_2(n, 4)$, $\mathcal{G}_3(n, 4)$, $|\mathcal{G}|=|\mathcal{F}|$, but
$s_{i, j}(\mathcal{G})\subseteq \mathcal{G}_2(n, 4)$.

By conditions, there exists  $x_0, x_1, x_2$ such that
$$
s_{i, j}(\mathcal{G})\subseteq \left\{G \in\binom{[n]}{4}: \{x_1,x_2\} \subseteq G\right\} \cup\left\{G \in\binom{[n]}{4}: x_0 \in G, G \cap \{x_1,x_2\} \neq \emptyset\right\}.
$$
Hence, $i\in \{x_0, x_1, x_2\}$ and $j\notin \{x_0, x_1, x_2\}$. By symmetry, we may assume that $x_0=i$. This implies that $|\mathcal{G}(i,j)|=|s_{i, j}(\mathcal{G})(i,j)|\leq \binom{n-2}{2}- \binom{n-4}{2}= 2n-7$.
By the maximality of $|\mathcal{G}|$, we have $|\mathcal{G}(\bar{i},\bar{j})|=\binom{n-4}{2}$. Observe that $\mathcal{G}(i,\bar{j}), \mathcal{G}(\bar{i},j) \subseteq\left\{A \in\binom{[n]\backslash \{i,j\}}{3}:  A\cap \{x_1, x_2\}\neq \emptyset\right\}$. Let
$
\mathcal{R}=\left\{A \in\binom{[n]\backslash \{i,j\}}{3}: \{x_1,x_2\} \subseteq A\right\}.
$
By the maximality of $|\mathcal{G}|$ and $\mathcal{G} \nsubseteq\mathcal{G}_2(n, 4)$, we may assume that $\mathcal{R}\subsetneq \mathcal{G}(i,\bar{j})$ and $\mathcal{R}\subsetneq \mathcal{G}(\bar{i},j)$. Note that $\mathcal{G}(i,\bar{j})\cap \mathcal{G}(\bar{i},j)\subseteq \mathcal{R}$.
So $(\mathcal{G}(i,\bar{j})\backslash\mathcal{R}) \cap(\mathcal{G}(i,\bar{j})\backslash\mathcal{R}) =\emptyset$.
Since $\mathcal{G}(i,\bar{j})$ and $\mathcal{G}(\bar{i},j)$ are cross-intersecting and $n-2\geq 7$, by Lemma \ref{Lem52}, we obtain
\begin{align*}
|\mathcal{F}|&=|\mathcal{G}(i,j)|+|\mathcal{G}(\bar{i},\bar{j})|+|\mathcal{G}(i,\bar{j})|+|\mathcal{G}(\bar{i}, j)|\\
&\leq 2n-7+\binom{n-4}{2}+ \binom{n-2}{3}-\binom{n-5}{3}-\binom{n-6}{2}-\binom{n-7}{1}+2\\
&<\binom{n-1}{3}-\binom{n-5}{3}-\binom{n-6}{2}-\binom{n-7}{1}+3.
\end{align*}

{\bf Case 2.}
 Suppose that $\mathcal{F}$ is not shifted, and there exists $\mathcal{G} \subseteq\binom{[n]}{4}$ obtained from $\mathcal{F}$ by repeated shifting operations such that $\mathcal{G}$  is neither EKR nor HM, and  $\mathcal{G} \nsubseteq \mathcal{J}_2(n, 4)$, $\mathcal{G}_2(n, 4)$, $\mathcal{G}_3(n, 4)$, $|\mathcal{G}|=|\mathcal{F}|$, but
$s_{i, j}(\mathcal{G})\subseteq \mathcal{G}_3(n, 4)$.

By conditions, there exist  $x_0, x_1, x_2, x_3$ such that
$$
s_{i, j}(\mathcal{G})\subseteq \left\{G \in\binom{[n]}{4}: \{x_1,x_2, x_3\} \subseteq G\right\} \cup\left\{G \in\binom{[n]}{4}: x_0 \in G, G \cap \{x_1,x_2, x_3\} \neq \emptyset\right\}.
$$
Hence, $i=x_0$, or $i\in \{x_1, x_2, x_3\}$ and $j \notin \{x_0, x_1, x_2, x_3\}$.

If $i=x_0$ and $j\in \{x_1, x_2, x_3\}$, let $j=x_1$. Then $\mathcal{G}(\bar{i},\bar{j})= \emptyset$ and $|\mathcal{G}(i,j)|\leq \binom{n-2}{2}$. Observe that $\mathcal{G}(i,\bar{j}), \mathcal{G}(\bar{i},j) \subseteq\left\{A \in\binom{[n]\backslash \{i,j\}}{3}:  A\cap \{x_2, x_3\}\neq \emptyset\right\}$. Let
$
\mathcal{R}=\left\{A \in\binom{[n]\backslash \{i,j\}}{3}: \{x_2,x_3\} \subseteq A\right\}.
$
By the maximality of $|\mathcal{G}|$ and $\mathcal{G} \nsubseteq\mathcal{G}_3(n, 4)$, we may assume that $\mathcal{R}\subsetneq \mathcal{G}(i,\bar{j})$ and $\mathcal{R}\subsetneq \mathcal{G}(\bar{i},j)$. Note that $\mathcal{G}(i,\bar{j})\cap \mathcal{G}(\bar{i},j)\subseteq \mathcal{R}$.
So $(\mathcal{G}(i,\bar{j})\backslash\mathcal{R}) \cap(\mathcal{G}(i,\bar{j})\backslash\mathcal{R}) =\emptyset$.
Since $\mathcal{G}(i,\bar{j})$ and $\mathcal{G}(\bar{i},j)$ are 
cross-intersecting and $n-2\geq 7$, by Lemma \ref{Lem52}, we obtain
\begin{align*}
|\mathcal{F}|&=|\mathcal{G}(i,j)|+|\mathcal{G}(i,\bar{j})|+|\mathcal{G}(\bar{i}, j)| \\ 
& \leq \binom{n-2}{2}+ \binom{n-2}{3}-\binom{n-5}{3}-\binom{n-6}{2}-\binom{n-7}{1}+2\\
&<\binom{n-1}{3}-\binom{n-5}{3}-\binom{n-6}{2}-\binom{n-7}{1}+3.
\end{align*}

If $i=x_0$ and $j\notin \{x_1, x_2, x_3\}$, then $|\mathcal{G}(i,j)|\leq 3n-12$ and  $|\mathcal{G}(\bar{i},\bar{j})|\leq n-5$. Observe that  $\mathcal{G}(i,\bar{j}), \mathcal{G}(\bar{i},j) \subseteq\left\{A \in\binom{[n]\backslash \{i,j\}}{3}:  A\cap \{x_1, x_2, x_3\}\neq \emptyset\right\}$ and $\mathcal{G}(i,\bar{j})\cap \mathcal{G}(\bar{i},j)\subseteq \{x_1,x_2,x_3\}$. Since $\mathcal{G}$ is not a subfamily of $\mathcal{G}_3(n, 4)$, we have that $\mathcal{G}(i,\bar{j})$ and
$\mathcal{G}(\bar{i},j)$ are non-empty, $\mathcal{G}(i,\bar{j})\neq \{x_1, x_2, x_3\}$ and $\mathcal{G}(\bar{i},j)\neq  \{x_1, x_2, x_3\}$.
 By the maximality of $|\mathcal{G}|$, we may assume that $\mathcal{G}(i,\bar{j})\cap \mathcal{G}(\bar{i},j)=\{x_1, x_2, x_3\}$. So  $|\mathcal{G}(i,\bar{j})|\geq 2$ and $|\mathcal{G}(\bar{i},j)|\geq 2$.
Note that $\mathcal{G}(i,\bar{j})$ and $\mathcal{G}(\bar{i},j)$ are 
cross-intersecting.
 Applying Lemma \ref{W231} to $\mathcal{G}(i,\bar{j})$ and $\mathcal{G}(\bar{i},j)$ yields
\begin{align*}
|\mathcal{F}|&=|\mathcal{G}(i,j)|+|\mathcal{G}(\bar{i},\bar{j})|+|\mathcal{G}(i,\bar{j})|+|\mathcal{G}(\bar{i}, j)|\\
&\leq 3n-12+n-5+\binom{n-2}{3}- \binom{n-5}{3}-\binom{n-6}{2}+1\\
&=\binom{n-2}{3}- \binom{n-5}{3}-\binom{n-6}{2}+4n-16 \\
&< \binom{n-1}{3}-\binom{n-5}{3}-\binom{n-6}{2}-\binom{n-7}{1}+3, 
\end{align*}
where the last inequality holds by direct computation.

If $i\in \{x_1, x_2, x_3\}$ and $j \notin \{x_0, x_1, x_2, x_3\}$, we may assume that $x_1=i$. Then $\mathcal{G}(i,j)\subseteq\{x_2,x_3\}\cup\{ \{x_0,a\}:a\in [n]\backslash\{i, j, x_0\}\}$ and
$
\mathcal{G}(\bar{i},\bar{j}) \subseteq\left\{G \in\binom{[n]\backslash \{i,j\}}{4}: x_0\in G, G\cap\{x_2,x_3\} \neq \emptyset\right\}
$.
So $|\mathcal{G}(i,j)|\leq n-2$ and  $|\mathcal{G}(\bar{i},\bar{j})|\leq \binom{n-3}{3}-\binom{n-5}{3}$.
Let
\begin{align*}
\mathcal{S}=\left\{S \in\binom{[n]\backslash \{i,j\}}{3}: x_0\in S, S\cap\{x_2,x_3\} \neq \emptyset\right\},~\mathcal{T}=\left\{T \in\binom{[n]\backslash \{i,j\}}{3}: x_0\in T\right\}.
\end{align*}
Observe that
$\mathcal{G}(i,\bar{j}), \mathcal{G}(\bar{i},j) \subseteq \mathcal{T}\cup \{\{x_2, x_3, a\}:a\in [n]\backslash\{i, j, x_2,x_3,x_0\}\}$.
Moreover, we have $\mathcal{G}(i,\bar{j})\cap \mathcal{G}(\bar{i},j)\subseteq \mathcal{S}$. By the maximality of $|\mathcal{G}|$ and $\mathcal{G} \nsubseteq\mathcal{G}_3(n, 4)$, we may assume that $\mathcal{S}\subsetneq \mathcal{G}(i,\bar{j})$ and 
$\mathcal{S}\subsetneq \mathcal{G}(\bar{i},j)$. 
Note that $(\mathcal{G}(i,\bar{j})\backslash\mathcal{S}) \cap(\mathcal{G}(i,\bar{j})\backslash\mathcal{S}) =\emptyset$, and both 
$\mathcal{G}(i,\bar{j})\backslash\mathcal{S} $ and $ \mathcal{G}(\bar{i},j) \backslash\mathcal{S}$ are subfamilies of 
$$
\left\{A \in\binom{[n]\backslash \{i,j\}}{3}: x_0\in A, A\cap\{x_2,x_3\} 
= \emptyset\right\}
\cup \{\{x_2, x_3, a\}:a\in [n]\backslash\{i, j, x_2,x_3,x_0\}\}
.$$
In adition,  $\mathcal{G}(i,\bar{j})\backslash\mathcal{S}$ and $\mathcal{G}(\bar{i},j)\backslash\mathcal{S}$ are cross-intersecting.

Next we consider the following four cases, and we claim that in all these cases, the inequality
$|\mathcal{G}(i,\bar{j})\backslash\mathcal{S}|+|\mathcal{G}(\bar{i},j)\backslash\mathcal{S}|\leq \binom{n-5}{2}+2$
holds for $n\geq 9$. Once we have established this inequality. Then
\begin{align*}
|\mathcal{F}|&=|\mathcal{G}(i,j)|+2|\mathcal{S}|+|\mathcal{G}(i,\bar{j})\backslash\mathcal{S}|+|\mathcal{G}(\bar{i},j)\backslash\mathcal{S}|+|\mathcal{G}(\bar{i}, j)|\\
&\leq n-2+2\left(\binom{n-3}{2}- \binom{n-5}{2}\right)+\binom{n-5}{2}+2+ \binom{n-3}{3}-\binom{n-5}{3}\\
&=\binom{n-1}{3}-\binom{n-5}{3}-\binom{n-6}{2}-\binom{n-7}{1}+2\\
&<\binom{n-1}{3}-\binom{n-5}{3}-\binom{n-6}{2}-\binom{n-7}{1}+3,
\end{align*}
proving the Case 2. Let us finish the proof by proving the claim.

(i) If there exist $\{x_2, x_3, u\}, \{x_2, x_3, v\}, \{x_2, x_3, w\}
\in \mathcal{G}(i,\bar{j})\backslash\mathcal{S}$ for three different elements $u, v, w\in [n]\backslash\{i, j, x_2,x_3,x_0\}$, then
$
\mathcal{G}(\bar{i},j) \backslash\mathcal{S}\subseteq\{\{x_2, x_3, a\}:a\in [n]\backslash\{i, j, x_2,x_3,x_0, u, v, w\}\}.
$
So $|\mathcal{G}(\bar{i},j)\backslash\mathcal{S}|\leq n-8$.

(ii) If there exist $\{x_2, x_3, u\}, \{x_2, x_3, v\}
\in \mathcal{G}(i,\bar{j})\backslash\mathcal{S}$ for $u\neq v\in [n]\backslash\{i, j, x_2,x_3,x_0\}$, then
$
\mathcal{G}(\bar{i},j) \backslash\mathcal{S}\subseteq\{\{x_2, x_3, a\}:a\in [n]\backslash\{i, j, x_2,x_3,x_0, u, v\}\}\cup \{x_0, u, v\}.
$
So $|\mathcal{G}(\bar{i},j)\backslash\mathcal{S}|\leq n-6$.

(iii) If there exists $\{x_2, x_3, u\}\in \mathcal{G}(i,\bar{j})\backslash\mathcal{S}$ for $u\in [n]\backslash\{i, j, x_2,x_3,x_0\}$, then
$
\mathcal{G}(\bar{i},j) \backslash\mathcal{S}$ $\subseteq$ $\{\{x_2, x_3, $ $a\} : a\in [n]\backslash\{i, j, x_2,x_3,x_0, u\}\}\cup\{ \{x_0, u, a\}: a\in [n]\backslash\{i, j, x_2,x_3,x_0, u\}\}.
$
So $|\mathcal{G}(\bar{i},j)\backslash\mathcal{S}|\leq 2n-12$.

(iv) If there is no such $\{x_2, x_3, u\}\in \mathcal{G}(i,\bar{j})\backslash\mathcal{S}$ for $u\in [n]\backslash\{i, j, x_2,x_3,x_0\}$, then $\mathcal{G}(i,\bar{j}) \backslash\mathcal{S}\neq \emptyset$ implies that there exists $\{x_0, y, z\}\in \mathcal{G}(i,\bar{j}) \backslash\mathcal{S}$ for some $y\neq z \in  [n]\backslash\{i, j, x_2,x_3,x_0\}$. So
$
\mathcal{G}(\bar{i},j) \backslash\mathcal{S}\subseteq\{\{x_2, x_3, a\}:a\in \{y,z\}\}\cup\left(\left\{A \in\binom{[n]\backslash \{i,j\}}{3}: x_0\in A, A\cap\{x_2,x_3\} = \emptyset\right\}\backslash \{x_0, y, z\}\right).
$
Therefore, we have $|\mathcal{G}(i,\bar{j})\backslash\mathcal{S}|+|\mathcal{G}(\bar{i},j)\backslash\mathcal{S}|\leq \binom{n-5}{2}+2$.

Returning to case (i), since $\mathcal{G}(\bar{i},j) \backslash\mathcal{S}
\neq \emptyset$, by symmetry and case (iii), we have  $|\mathcal{G}(i,\bar{j})\backslash\mathcal{S}|\leq 2n-12$. Thus $|\mathcal{G}(i,\bar{j})\backslash\mathcal{S}|+|\mathcal{G}(\bar{i},j)\backslash\mathcal{S}|\leq 3n-20<\binom{n-5}{2}+2$.

 For case (ii), we may assume that case (i) is not true.
If there exists $\{x_2, x_3, a\}\in \mathcal{G}(\bar{i},j)\backslash\mathcal{S}$ for some $a\in [n]\backslash\{i, j, x_2,x_3,x_0, u, v\}$, then by symmetry and case (iii), note that case (i) is not true, we have $|\mathcal{G}(i,\bar{j})\backslash\mathcal{S}|\leq 2+n-6=n-4$. Then $|\mathcal{G}(i,\bar{j})\backslash\mathcal{S}|+|\mathcal{G}(\bar{i},j)\backslash\mathcal{S}|\leq 2n-10\leq \binom{n-5}{2}+2$. Otherwise by symmetry and case (iv), we also have $|\mathcal{G}(i,\bar{j})\backslash\mathcal{S}|+|\mathcal{G}(\bar{i},j)\backslash\mathcal{S}|\leq \binom{n-5}{2}+2$.

We have completed the proofs of cases (i), (ii), (iv).
For case (iii), we may assume that cases (i) and (ii) are not true.
Furthermore, by symmetry and cases (i), (ii) (iv), we may assume that there is only one
$\{x_2, x_3, a\}\in \mathcal{G}(\bar{i},j)\backslash\mathcal{S}$ for some $a\in [n]\backslash\{i, j, x_2,x_3,x_0, u\}$.
Then $|\mathcal{G}(\bar{i},j)\backslash\mathcal{S}|\leq 1+n-6=n-5$ and $|\mathcal{G}(i,\bar{j})\backslash\mathcal{S}|\leq n-5$.
Therefore, we obtain $|\mathcal{G}(i,\bar{j})\backslash\mathcal{S}|+|\mathcal{G}(\bar{i},j)\backslash\mathcal{S}|\leq 2n-10\leq \binom{n-5}{2}+2$.
This completes the proof.$\hfill \square$

\end{document}